\RequirePackage{fix-cm}

\documentclass[final,envcountsect,smallextended]{svjour3}

\smartqed

\usepackage[greek,USenglish]{babel}
\usepackage[T1]{fontenc}
\usepackage[latin10]{inputenc}
\usepackage{amsmath,amssymb} 
\usepackage{lmodern}
\usepackage{upgreek}
\usepackage[obeyspaces]{url}
\usepackage{graphicx,tikz,pgfplots}
\usepackage{centernot}
\usepackage{array}
\usepackage{enumitem}
\usepackage[bookmarks=true,plainpages=false,colorlinks=true,citecolor=green!80!black,linkcolor=red!70!black,filecolor=magenta,urlcolor=magenta,breaklinks,pdfauthor={Thomas Jahn, Yaakov Kupitz, Horst Martini, Christian Richter},pdftitle={Minsum Location Extended to Gauges and to Convex Sets}]{hyperref}
\usepackage[babel]{microtype}
\usepackage{calc,mathtools}
\usepackage{cite}

\usetikzlibrary{calc,intersections,through,backgrounds,arrows,patterns,shapes.geometric,shapes.misc}

\newcommand{\fall}{\:\forall\:}
\newcommand{\ex}{\:\exists\:}
\newcommand{\ZZ}{\mathbb{Z}}
\newcommand{\EE}{\mathbb{E}}
\newcommand{\RR}{\mathbb{R}}
\newcommand{\restr}[2]{\left.#1\right|_{#2}}
\newcommand{\abs}[1]{\left\lvert#1\right\rvert}
\newcommand{\abss}[1]{\lvert#1\rvert}
\newcommand{\mnorms}[1]{\lVert#1\rVert}  
\newcommand{\mnorm}[1]{\left\lVert#1\right\rVert}
\newcommand{\setn}[1]{\left\{#1\right\}}

\newcommand{\setcond}[2]{\left\{#1\: :\: #2\right\}}
\newcommand{\defeq}{\mathrel{\mathop:}=}
\newcommand{\zB}{e.g., }
\newcommand{\dah}{i.e., }
\newcommand{\lr}[1]{\!\left(#1\right)}
\newcommand{\clseg}[2]{\left[#1,#2\right]}
\newcommand{\mc}[2]{B(#1,#2)}
\newcommand{\ms}[2]{S(#1,#2)}
\newcommand{\mcg}[3]{B_{#1}(#2,#3)}
\newcommand{\msg}[3]{S_{#1}(#2,#3)}
\newcommand{\cRR}{\overline{\mathbb{R}}}
\newcommand{\p}{\partial}
\newcommand{\skpr}[2]{\left\langle#1 \,\middle\vert\, #2\right\rangle}
\newcommand{\norel}{\mathrel{\phantom{=}}}
\newcommand{\boxx}{\ensuremath{\mbox{\small$\,\square\,$}}}
\newcommand{\enquote}[1]{``#1''}

\newlength{\mywidth}
\newcommand{\gleich}[1][\rule{0pt}{0pt}]{\stackrel{\text{\makebox[\mywidth]{\centering #1}}}{=}}
\newcommand{\groessergleich}[1][\rule{0pt}{0pt}]{\stackrel{\text{\makebox[\mywidth]{\centering #1}}}{\geq}}
\newcommand{\groesser}[1][\rule{0pt}{0pt}]{\stackrel{\text{\makebox[\mywidth]{\centering #1}}}{>}}

\DeclareMathOperator{\lin}{lin}
\DeclareMathOperator{\aff}{aff}
\DeclareMathOperator{\co}{conv}

\DeclareMathOperator{\bd}{bd}

\DeclareMathOperator{\ri}{relint}

\DeclareMathOperator{\ext}{ext}
\DeclareMathOperator{\Proj}{proj}
\DeclareMathOperator{\ft}{ft}
\DeclareMathOperator{\dist}{dist}

\DeclareMathOperator{\nor}{nor}
\let \eps \varepsilon
\let \piup \uppi

\AtBeginDocument{%
  \setlength{\oddsidemargin}{\dimexpr(\paperwidth-\textwidth)/2-1in}%
  \setlength{\evensidemargin}{\oddsidemargin}%
  \setlength{\topmargin}{%
    \dimexpr(\paperheight-\textheight)/2-\headheight-\headsep-1in}%
}

\allowdisplaybreaks[1]
\begin{document}
\parindent 0pt
\title{Minsum Location Extended to Gauges and to Convex Sets\footnote{submitted to Journal of Optimization Theory and Applications}}


\author{Thomas Jahn \and Yaakov S. Kupitz \and Horst Martini \and Christian Richter}


\institute{T. Jahn (corresponding author) \at
              Faculty of Mathematics, Chemnitz University of Technology, 09107 Chemnitz, Germany\\
              \email{thomas.jahn@mathematik.tu-chemnitz.de}
           \and
           Y. S. Kupitz \at
              Institute of Mathematics, The Hebrew University of Jerusalem, Jerusalem, Israel\\
               \email{kupitz@math.huji.ac.il}
           \and
           H. Martini \at
              Faculty of Mathematics, Chemnitz University of Technology, 09107 Chemnitz, Germany\\
               \email{horst.martini@mathematik.tu-chemnitz.de}
           \and
           C. Richter \at
              Institute of Mathematics, Friedrich Schiller University, 07737 Jena, Germany\\
               \email{christian.richter@uni-jena.de}
}

\date{Received: date / Accepted: date}

\maketitle

\begin{abstract}
One of the oldest and richest problems from continuous location science is the famous Fermat--Torricelli problem, asking for the unique point in Euclidean space that has minimal distance sum to $n$ given (non-collinear) points. Many natural and interesting generalizations of this problem were investigated, \zB by extending it to non-Euclidean spaces and modifying the used distance functions, or by generalizing the configuration of participating geometric objects. In the present paper, we extend the Fermat--Torricelli problem in a two-fold way: more general than for normed spaces, the unit balls of our spaces are compact convex sets having the origin as interior point (but without symmetry condition), and the $n$ given objects can be general convex sets (instead of points). We combine these two viewpoints, and the presented sequence of new theorems follows in a comparing sense that of theorems known for normed spaces. Some of these results holding for normed spaces carry over to our more general setting, and others not. In addition, we present analogous results for related questions, like, \zB for Heron's problem. And finally we derive a collection of results holding particularly for the Euclidean norm.
\end{abstract}

\section{Introduction}
\label{chap:introduction}
The classical Fermat-Torricelli problem asks for the unique point minimizing the distance sum to finitely many non-collinear points in $d$-dimensional Euclidean space. It is, regarding the variety of contributions and contributors, one of the richest problems from continuous location science. Going back to the 17th century, it still creates new research problems. At its historically first step, it was connected with famous names like R. Descartes, P. de Fermat, E. Torricelli, V. Viviani, B. Cavalieri, and E. W. von Tschirnhaus. Later on, mathematicians like J. Bertrand, C. F. Gauss, J. Steiner, L. Lindel\"{o}f, R. Sturm, J. Hadamard, G. Polya, H. W. Kuhn, P. Erd\H{o}s and many others added results more related to modern branches of mathematics, like (convex) optimization, approximation theory, functional analysis, algebraic geometry, convex analysis, computational geometry etc. (a comprehensive representation is given in \cite[Chapter II]{BoltyanskiMaSo1999}). More recently, deeper generalizations of this problem were added, for example extensions to non-Euclidean spaces of different types, and generalizations of the participating geometric configuration. Extensions to normed spaces (i.e., to real, finite dimensional Banach spaces; see \cite{DurierMi1985}, \cite{MichelotLe1987}, and \cite{MartiniSwWe2002}) and replacements of the participating points by hyperplanes or spheres (cf., \zB \cite{MartiniSc2001} and \cite{KoernerMaSc2012}) yield especially interesting and geometrically rich approaches and algorithms. Our goal here is to continue this line of research by generalizing the latter two viewpoints at the same time, thus reaching a  more general step in the combined sense. Namely, first we extend the basic theory on this problem from normed spaces to generalized Minkowski spaces, having arbitrary convex bodies as unit balls which no longer need to be centrally symmetric; they create general convex distance functions (gauges). And second, we generalize the geometric properties of the participating given objects: they are no longer points or hyperplanes, but arbitrary convex sets. Related topics are also discussed in \cite[pp.~146--168]{MordukhovichNa2013}.

To do this, we give in our second section also novel extensions of several common notions from Banach space theory to gauges, \zB by introducing generalized norming functionals. Such notions are fundamental for the geometric description of the solution sets of our location problems. It turns out that for introducing generalized norming functionals, a correspondingly generalized version of the Hahn--Banach theorem is necessary. So we also use a version of the Hahn--Banach theorem extended to gauges.

In the third section we derive results on the Fermat--Torricelli problem for finite point sets with respect to  gauges, mainly generalizing the sequence of theorems presented for normed spaces in \cite{MartiniSwWe2002} (see also \cite{DurierMi1985}). It turns out that some of the results from \cite{MartiniSwWe2002} directly carry over to generalized Minkowski spaces, and some not. Denoting by $\ft(P)$ the solution set of the Fermat--Torricelli problem for a non-collinear finite point set $P$ in a generalized Minkowski space, we clarify the cases when $\ft(P)$ is a singleton or a polytope. Strict convexity of the unit ball is now only a necessary criterion for the property that $\ft(P)$ is a singleton, for every $P$ as above. We present also characterizations of normed spaces within the family of generalized Minkowski spaces via an extension of Menger's notion of $d$-segments (often needed for metrical problems in normed spaces; see \cite{Menger1928} and \cite[Chapter II]{BoltyanskiMaSo1997}). We show that in this more general setting $\ft(P)$ can be represented as intersection of certain cones determined by the boundary structure of the convex unit balls. In this section we also investigate the boundary structure of sublevel sets for the generalized Fermat--Torricelli problem; their geometry again depends on the boundary structure of the unit balls.

In the fourth section we combine the approach via gauges with the replacement of the given finite point set by a given finite family of arbitrary convex sets, even considering finitely many respective gauges. Using generalized norming functionals and a corresponding extension of distance functions to gauges, we derive similar results as they exist for given point sets, and some interesting additional observations are obtained, too. For example, we investigate also Heron's problem in our generalized setting. In the final fifth section, further theorems holding only for the Euclidean norm are derived (the given objects are still arbitrary convex sets).

\section{Generalized Minkowski Spaces and Basics from Convex Analysis}
\label{chap:basics}
Let $X$ be a finite-dimensional real vector space. The \emph{closed line segment} between $x,y \in X$ is denoted by $[x,y]$. The symbol $\ri(A)$ stands for the \emph{relative interior} of a subset $A \subseteq X$, i.e., for the interior of $A$ in the natural topology of the affine span of $A$. For finite subsets $\setn{x_1,\ldots,x_n}$ of $X$, we write $\lin\setn{x_1,\ldots,x_n}$ and $\aff\setn{x_1,\ldots,x_n}$ for their \emph{linear} and \emph{affine hulls}, respectively. Let $\RR_+=[0,+\infty[$ and $\RR_{++}=\mathopen{]}0,+\infty\mathclose{[}$. As usual in convex analysis, there will be functions $X\to \cRR$, where $\cRR=\RR\cup\setn{+\infty,-\infty}$ is the \emph{extended real line} with the conventions $0(+\infty)\defeq +\infty$, $0(-\infty)\defeq 0$, $(+\infty)+(-\infty)\defeq +\infty$.
\begin{definition}\label{def-gauge}
A \emph{gauge} on $X$ is a functional $\gamma: X\to \RR_+$ satisfying the conditions
\begin{enumerate}[label={(\alph*)},leftmargin=0.3cm,align=left]
\item{$\gamma(x)=0\Longrightarrow x=0$ \, for all $x \in X$,\label{linefree}}
\item{$\gamma(\lambda x)=\lambda \gamma(x)$ \, for all $x \in X$ and $\lambda\in\RR_+$ \, \emph{(positive homogeneity)},\label{homogeneous}}
\item{$\gamma(x+y)\leq \gamma (x)+\gamma(y)$ \, for all $x,y \in X$ \, \emph{(subadditivity, triangle inequality)}\label{subadditive},}
\end{enumerate}
see, \zB \cite{Rockafellar1972, Plastria1992b}. The pair $(X,\gamma)$ is called a \emph{generalized Minkowski space}, and with $\gamma(x)=\gamma(-x)$ for all $x\in X$ it is a \emph{normed} (or \emph{Minkowski}) \emph{space}. The \emph{ball with radius $\lambda \in \RR_{+}$ and whose center} (not meant regarding shape) \emph{is determined at $x\in X$} is the set
\begin{equation*}
\mcg{\gamma}{x}{\lambda}=\setcond{y\in X}{\gamma(y-x)\leq \lambda}.
\end{equation*}
The respective \emph{sphere} is given by
\begin{equation*}
\msg{\gamma}{x}{\lambda}=\setcond{y\in X}{\gamma(y-x)=\lambda}.
\end{equation*}
If $\gamma$ is clear from the context, we omit it from the notation. If $\gamma(x)=1$, then $x$ is a \emph{unit vector}.
\end{definition}
\begin{example}Let $(X,\gamma)$ be a generalized Minkowski space. 
\begin{enumerate}[label={(\alph*)},leftmargin=0.3cm,align=left]
\item{The \emph{unit ball} $\mc{0}{1}$ is a compact, convex set having the origin as interior point. Conversely, if $B\subseteq X$ is a compact, convex set having the origin as interior point, then $\gamma_B:X\to\RR$, $\gamma_B(x)=\inf\setcond{\lambda\in \RR_+}{x\in \lambda B}$, defines a gauge on $X$.}
\item{The \emph{opposite gauge} $\tilde{\gamma}:X\to\RR$, $\tilde{\gamma}(x)=\gamma(-x)$, of $\gamma$ defines another gauge on $X$.}
\end{enumerate}
\end{example}

The concept of the dual space of a normed space is very important in classical functional analysis. The right extension of this notion for generalized Minkowski spaces $(X,\gamma)$ is given by the cone of linear and upper semi\-con\-tin\-u\-ous functions $\phi:X\to\RR$. Since we are concerned with finite-di\-men\-sional vector spaces, the topology generated by $\gamma$ satisfies the $T_1$ separation axiom \cite[Proposition~1.1.8]{Cobzas2013} and, by \cite[Theorem~9]{GarciaRaffi2005}, is the Euclidean topology. Thus all linear functionals $\phi:X\to\RR$ are continuous.
\begin{definition}
The \emph{dual space} of the vector space $X$ is the vector space $X^\ast$ of linear functionals $\phi: X\to \RR$. For $\phi \in X^\ast$ and $x\in X$, we shall write $\skpr{\phi}{x}$ for $\phi(x)$.
\end{definition}
The concept of the dual norm is replaced by the polar function.
\begin{definition}
The \emph{polar function} of a gauge $\gamma: X\to \RR_+$ is given by
\begin{equation*}
\gamma^\circ:X^\ast \to \RR_+,\,\gamma^\circ(\phi)\defeq \inf\setcond{\lambda\in\RR_{++}}{\skpr{\phi}{x}\leq \lambda \gamma (x) \fall x\in X}.
\end{equation*}
\end{definition}
From the definition it follows that gauges satisfy the Cauchy--Schwarz-like inequalities
\begin{equation}
-\gamma^\circ(-\phi)\gamma (x)\leq \skpr{\phi}{x}\leq \gamma^\circ(\phi) \gamma (x)\label{eq:cauchy-schwarz}
\end{equation} for all $\phi\in X^\ast$, $x\in X$. Other representations of the polar function are
\begin{align*}
\gamma^\circ(\phi)&=\sup\setcond{\frac{\skpr{\phi}{y}}{\gamma(y)}}{y\in X, y\neq 0}\\
&=\sup\setcond{\skpr{\phi}{y}}{y\in X,\gamma(y)=1}\\
&=\sup\setcond{\skpr{\phi}{y}}{y\in X,\gamma(y)\leq 1},
\end{align*}
see \cite[\textsection\,15]{Rockafellar1972} and \cite[Proposition~2.1.7]{Cobzas2013}. Note that $(X^\ast,\gamma^\circ)$ is a generalized Minkowski space. The polar gauge $\gamma^\circ$ can also be viewed as the support function of the unit ball of $\gamma$.

\begin{definition}
The \emph{support function} of a set $K\subseteq X$ is given by 
\begin{equation*}
h(\cdot,K):X^\ast \to \cRR,\, h(\phi,K)\defeq \sup\setcond{\skpr{\phi}{x}}{x\in K}.
\end{equation*}
The \emph{polar set} of $K$ is $K^\circ \defeq\setcond{\phi\in X^\ast}{h(\phi,K)\leq 1}$. 
\end{definition}

There is an intimate relationship between a gauge $\gamma:X\to\RR$ and its opposite $\tilde{\gamma}:X\to\RR$. When combined with polarity, we obtain the following formulas.
\begin{proposition}[see {\cite[Theorem~15.1]{Rockafellar1972}}]
Let $(X,\gamma)$ be a generalized Minkowski space. Then
\begin{enumerate}[label={(\alph*)},leftmargin=0.3cm,align=left]
\item{$\mcg{\tilde{\gamma}}{0}{1}=-\mcg{\gamma}{0}{1}$,}
\item{$\mcg{\gamma^\circ}{0}{1}=\mcg{\gamma}{0}{1}^\circ$,}
\item{$(\tilde{\gamma})^\circ = (\gamma^\circ)\tilde{\;}$,}
\item{$(-\mcg{\gamma}{0}{1})^\circ = -\mcg{\gamma}{0}{1}^\circ$.}
\end{enumerate}
\end{proposition}

The \emph{Hahn--Banach theorem} is a link between functional analysis and convex geometry. Its numerous appearances include norm-preserving extension of linear functions and separation of convex sets by hyperplanes. We give the version appropriate for generalized Minkowski spaces.
\begin{theorem}[{see \cite[Theorem~2.2.2]{Cobzas2013}}]\label{thm:hahn-banach}
Let $(X,\gamma)$ be a generalized Min\-kow\-ski space.
\begin{enumerate}[label={(\alph*)},leftmargin=0.3cm,align=left]
\item{If $Y$ is a subspace of $X$ and $\phi_0:Y\to \RR$ is a linear functional on the generalized Minkowski space $(Y,\restr{\gamma}{Y})$, then there exists a linear functional $\phi:X\to \RR$ such that $\restr{\phi}{Y}=\phi_0$ and $\gamma^\circ(\phi)=\gamma^\circ(\phi_0)$.}
\item{If $x \in X \setminus \setn{0}$, then there exists a linear functional $\phi:X\to\RR$ such that $\gamma^\circ(\phi)=1$ and $\skpr{\phi}{x}=\gamma(x)$.\label{existence_norming_func}}
\end{enumerate}
\end{theorem}

The following result is another version of $(X^\ast)^\ast \cong X$ and $(\gamma^\circ)^\circ =\gamma$, see \cite[Korollar III.1.7]{Werner2011} for the special case of norms.
\begin{lemma}
In any generalized Minkowski space, we have 
\begin{equation*}
\gamma(x)=\max\setcond{\skpr{\phi}{x}}{\phi\in X^\ast, \gamma^\circ(\phi)\leq 1}
\end{equation*}
for all $x\in X$.
\end{lemma}
\begin{proof} Fix $x\in X$.
Taking the supremum over $\phi\in X^\ast$ with $\gamma^\circ(\phi)\leq 1$ in \eqref{eq:cauchy-schwarz}, we obtain 
\begin{align*}
\gamma(x)&=\sup\setcond{\gamma^\circ(\phi)\gamma(x)}{\phi\in X^\ast, \gamma^\circ(\phi)\leq 1}\\
&\geq \sup\setcond{\skpr{\phi}{x}}{\phi\in X^\ast, \gamma^\circ(\phi)\leq 1}.
\end{align*}
By Theorem~\ref{thm:hahn-banach}\ref{existence_norming_func}, there exists a functional $\phi_0\in X^\ast$ such that $\gamma^\circ(\phi_0)=1$ and $\skpr{\phi_0}{x}=\gamma(x)$, \dah $\gamma(x)=\skpr{\phi_0}{x}\leq \sup\setcond{\skpr{\phi}{x}}{\phi\in X^\ast, \gamma^\circ(\phi)\leq 1}$.\qed
\end{proof}
We still need the notions of convex functions and subdifferentials.
\begin{definition}
\begin{enumerate}[label={(\alph*)},leftmargin=0.3cm,align=left]
\item{A function $f:X\to \cRR$ is \emph{convex} iff 
\begin{equation*}
f(\lambda x+ (1-\lambda) y) \leq \lambda f(x)+(1-\lambda) f(y)
\end{equation*} for all $x,y\in X$ and for all $\lambda\in [0,1]$.}
\item{The \emph{conjugate function} of $f:X\to\cRR$ is given by 
\begin{equation*}
f^\ast:X^\ast\to \cRR,\,f^\ast(\phi)=\sup\setcond{\skpr{\phi}{x}-f(x)}{x\in X}.
\end{equation*}
}
\item{The \emph{subdifferential} of a function $f:X\to\cRR$ at a point $x\in X$ with $f(x)\in\RR$ is the set $\p f(x)=\setcond{\phi\in X^\ast}{f(y)-f(x)\geq \skpr{\phi}{y-x}\fall y\in X}$. If $f(x)\notin \RR$, we set $\p f(x)\defeq \emptyset$.}
\item{A functional $\phi\in X^\ast$ is called a \emph{$\gamma$-norming functional} for $x\in X$ iff $\gamma^\circ(\phi)=1$ and $\skpr{\phi}{x}=\gamma(x)$.}
\end{enumerate}
\end{definition}
The existence of $\gamma$-norming functionals is provided by Theorem~\ref{thm:hahn-banach}\ref{existence_norming_func}. Plastria \cite{Plastria1992b} gives a subdifferential formula for gauges via $\gamma$-norming functionals. 
\begin{lemma}\label{lem:subdiff_norm}
Let $\gamma$ be a gauge on $X$. Then $\gamma$ is convex and
\begin{equation*}
\p \gamma (x)=\begin{cases}\setcond{\phi\in X^\ast}{\gamma^\circ(\phi)\leq 1},&x=0,\\\setcond{\phi\in X^\ast}{\gamma^\circ(\phi)=1, \skpr{\phi}{x}=\gamma(x)},&x\neq 0.\end{cases}
\end{equation*}
\end{lemma}


\section{Finitely Many Points in Generalized Minkowski Spaces}
\label{chap:points_minkowski}
The straightforward generalization of the famous Fermat--Torricelli problem for generalized Minkowski spaces $(X,\gamma)$ is the convex optimization problem
\begin{equation}
\inf_{x\in X}\sum_{i=1}^n \gamma(p_i-x), \label{eq:fermat-torricelli}
\end{equation}
where $P=\setn{p_1,\ldots,p_n}$ is a given set of $n\geq 1$ distinct points. The \emph{Fermat--Torricelli locus}, \dah the solution set of the Fermat--Torricelli problem for $P$, will be denoted by $\ft(P)$.

\subsection{Coincidences with the Norm Setting}
\label{chap:coincidences}
In this subsection we collect results that hold for gauges in the same way as for norms. In almost each case, the analogous statements for norms from \cite{MartiniSwWe2002} are cited in brackets.
\begin{proposition}\label{prop:ft_compact_convex}
The Fermat--Torricelli locus of any finite set is always non-empty, closed, bounded, and convex.
\end{proposition}
\begin{proof}
Since $f:X\to\RR$, $f(x)= \sum_{i=1}^n \gamma(p_i-x)$, is bounded from below by $0$, there exists $\alpha\defeq \inf\setcond{f(x)}{x\in X}\in\RR_+$.  The set 
\begin{equation*}
A\defeq\setcond{x\in X}{f(x)\leq \alpha +1}
\end{equation*}
is bounded (with respect to the gauge $\gamma$ and, by \cite[Theorem~9]{GarciaRaffi2005}, with respect to the Euclidean norm). Sublevel sets of convex and lower semicontinuous functions are convex and closed. Hence, $f$ attains its minimum on $A$ (which is $\alpha$), and the solution set $\setcond{x\in X}{f(x)=\alpha}=\setcond{x\in X}{f(x)\leq \alpha}$ is therefore non-empty, closed, bounded, and convex.\qed
\end{proof}
\begin{theorem}[{see \cite[Theorem~3.1]{MartiniSwWe2002}}]\label{thm:fermat-torricelli}
Let $f=\sum_{i=1}^n \gamma(p_i-\cdot)$ be the Fermat--Torricelli objective function.
\begin{enumerate}[label={(\alph*)},leftmargin=0.3cm,align=left]
\item{If $\bar{x}\in X\setminus\setn{p_1,\ldots,p_n}$, then $\bar{x}$ is a minimum point of $f$ if and only if, for all $i\in\setn{1,\ldots,n}$, there exists a $\gamma$-norming functional $\phi_i\in X^\ast$ of $p_i-\bar{x}$ such that $\sum_{i=1}^n \phi_i=0$.\label{thm_part_a}}
\item{A point $p_j$ ($1\leq j\leq n$) is a minimum point of $f$ if and only if, for all $i\in\setn{1,\ldots,n}$, $i\neq j$, there exists a $\gamma$-norming functional $\phi_i$ of $p_i-p_j$ such that $\gamma^\circ\lr{-\sum_{i\neq j}\phi_i}\leq 1$.\label{thm_part_b}}
\end{enumerate}
\end{theorem}
\begin{proof}
Using Lemma~\ref{lem:subdiff_norm}, the theorem is an immediate consequence of the fact that $\bar{x}$ is a minimum point of $f$ if and only if 
\begin{equation*}
0\in\p \lr{\sum_{i=1}^n \gamma(p_i-\cdot)}(\bar{x})= \sum_{i=1}^n \p \gamma(p_i-\cdot)(\bar{x}).
\end{equation*}
(Note that the functions $\gamma(p_i-\cdot)$ are real-valued and convex. Hence the subdifferential of the sum is the Minkowski sum of the subdifferentials, see \cite[Corollary~16.39]{BauschkeCo2011}.)

In both parts of the theorem, the `$\Leftarrow$' direction can also be shown directly in the same manner as for \cite[Theorem~2.1]{KupitzMaSp2013}.\\
\ref{thm_part_a} `$\Leftarrow$': Let $x \in X$. We have
\begin{align*}
f(x) &=\sum_{i=1}^n \gamma(p_i-x)\\
&\stackrel{\star}{\geq} \sum_{i=1}^n \skpr{\frac{\phi_i}{\gamma^\circ(\phi_i)}}{p_i-x}\\
&\stackrel{\star\star}{=} \sum_{i=1}^n \skpr{\phi_i}{p_i-x}\\
&=\sum_{i=1}^n \skpr{\phi_i}{p_i-\bar{x}+\bar{x}-x}\\
&=\sum_{i=1}^n \skpr{\phi_i}{p_i-\bar{x}}+\sum_{i=1}^n \skpr{\phi_i}{\bar{x}-x}\\
&\stackrel{\star\star}{=} \sum_{i=1}^n\gamma(p_i-\bar{x})+\skpr{\sum_{i=1}^n \phi_i}{\bar{x}-x}\\
&=f(\bar{x}).
\end{align*}

\ref{thm_part_b} `$\Leftarrow$': Let $x\in X$. We have
\begin{align*}
f(p_j)&=\sum_{i\neq j}\gamma(p_i-p_j)\\
&\stackrel{\star\star}{=} \sum_{i\neq j} \skpr{\phi_i}{p_i-p_j}\\
&=\sum_{i\neq j} \skpr{\phi_i}{p_i-x}+\sum_{i\neq j} \skpr{-\phi_i}{p_j-x}\\
&\stackrel{\star}{\leq} \sum_{i\neq j} \gamma^\circ(\phi_i)\gamma(p_i-x)+\gamma^\circ\lr{\sum_{i\neq j}-\phi_i}\gamma(p_j-x)\\
&\leq \sum_{i\neq j}\gamma(p_i-x)+ \gamma(p_j-x)\\
&=f(x).
\end{align*}
The relations $\stackrel{\star}{\leq}$ follow from the definition of the polar norm, which yields the Cauchy--Schwarz-like inequalities \eqref{eq:cauchy-schwarz}. Furthermore, the relations $\stackrel{\star\star}{=}$ hold true, since the functionals $\phi_i$ are assumed to be $\gamma$-norming functionals of $p_i-\bar{x}$.\qed
\end{proof}

\begin{definition}\label{def:exposedness_cones}
\begin{enumerate}[label={(\alph*)},leftmargin=0.3cm,align=left]
\item{An \emph{exposed face} of a closed convex set $K\subseteq X$ is the intersection of $K$ with one of its \emph{supporting hyperplanes} 
\begin{equation*}
\setcond{y\in X}{\skpr{\phi}{y}=h(\phi,K)},
\end{equation*}
whenever $\phi\in X^\ast$ obeys $h(\phi,K)<+\infty$. A point $x\in X$ is an \emph{exposed point} of $K$ iff $\setn{x}$ is an exposed face of $K$.\label{exposedness}}
\item{Given a unit functional $\phi \in X^\ast$ (\dah $\gamma^\circ(\phi)=1$) and a point $x\in X$, define the cone
\begin{equation*}
C(x,\phi)=x-\setcond{z\in X}{\skpr{\phi}{z}=\gamma(z)},
\end{equation*}
\dah $C(x,\phi)$ is the translate by $x$ of the rays from the origin through the exposed face $\phi^{-1}(-1)\cap (-\mcg{\gamma}{0}{1})$ of the unit ball $-\mcg{\gamma}{0}{1}$ of $(X,\tilde{\gamma})$.}
\end{enumerate}
\end{definition}
\begin{proposition}[{see \cite[Theorem~3.2]{MartiniSwWe2002}}]\label{prop:ft_locus_cone_intersection}
In any generalized Minkowski space $(X,\gamma)$ with finite given subset $P=\setn{p_1,\ldots,p_n}$, suppose that we are given $\bar{x}\in \ft(P)\setminus P$. Let $\phi_i$ be a $\gamma$-norming functional of $p_i-\bar{x}$ for each $i\in\setn{1,\ldots,n}$, such that $\sum_{i=1}^n \phi_i =0$. Then
\begin{equation*}
\ft(P)= \bigcap_{i=1}^n C(p_i,\phi_i).
\end{equation*}
\end{proposition}
\begin{proof}
By definition
\begin{equation*}
x\in \bigcap_{i=1}^n C(p_i,\phi_i) \;\Longleftrightarrow\; \skpr{\phi_i}{p_i-x}=\gamma(p_i-x) \fall i\in\setn{1,\ldots,n}.
\end{equation*}
Thus, if $x\notin P$, we have that $x\in \bigcap_{i=1}^n C(p_i,\phi_i)$ if and only if, for each $i\in\setn{1,\ldots,n}$, $\phi_i$ is a $\gamma$-norming functional of $p_i-x$. This yields  $x\in\ft(P)$ by Theorem~\ref{thm:fermat-torricelli} and the assumption $\sum_{i=1}^n\phi_i =0$. On the other hand, if $x=p_j$ for some $j$, then $x\neq p_i$ for all $i\neq j$, and $x\in \bigcap_{i=1}^n C(p_i,\phi_i)$ implies that, for all $i\neq j$, $\phi_i$ is a $\gamma$-norming functional of $p_i-x$. This implies that $x\in \ft(P)$, by Theorem~\ref{thm:fermat-torricelli} and $\gamma^\circ\lr{-\sum_{i\neq j}\phi_i}=\gamma^\circ(\phi_j)=1$. Thus $\bigcap_{i=1}^n C(p_i,\phi_i)\subseteq \ft(P)$. Conversely, if $x\in \ft(P)$, then 
\begin{align*}
\sum_{i=1}^n \skpr{\phi_i}{p_i-x}&=\sum_{i=1}^n \skpr{\phi_i}{p_i-\bar{x}+\bar{x}-x}\\
&=\sum_{i=1}^n \skpr{\phi_i}{p_i-\bar{x}}+\sum_{i=1}^n \skpr{\phi_i}{\bar{x}-x}\\
&=\sum_{i=1}^n \gamma(p_i-\bar{x})\\
&=\sum_{i=1}^n \gamma(p_i-x),
\end{align*}
and hence each Cauchy--Schwarz inequality $\skpr{\phi_i}{p_i-x}\leq\gamma(p_i-x)$ must hold as an equality. In other words, $\phi_i$ is a $\gamma$-norming functional for $p_i-x$, \dah $x\in\bigcap_{i=1}^n C(p_i,\phi_i)$.\qed
\end{proof}
\begin{corollary}[{see \cite[Corollary~3.2]{MartiniSwWe2002}}]
If all exposed faces of the unit ball $\mc{0}{1}$ of a generalized Minkowski space $(X,\gamma)$ are polytopes, then the Fermat--Torricelli locus of every set $P=\setn{p_1,\ldots,p_n} \subseteq X$ is a convex polytope, that may have empty interior. In particular, this applies if $X$ is two-dimensional or if $\mc{0}{1}$ is a polytope.
\end{corollary}
\begin{proposition}[{see \cite[Proposition~3.1]{MartiniSwWe2002}}]
Let $\setn{p_0,\ldots,p_n}\subseteq X$, $n \geq 1$, and let $\lambda_1,\ldots,\lambda_n \in\RR_{++}$.
\begin{enumerate}[label={(\alph*)},leftmargin=0.3cm,align=left]
\item{If $p_0\in\ft(\setn{p_1,\ldots,p_n})$, then $p_0\in\ft(\setn{p_0,p_1,\ldots,p_n})$.}
\item{If $p_0\in\ft(\setn{p_1,\ldots,p_n})$, then
\begin{equation*}
p_0\in \ft(\setn{p_0+\lambda_1(p_1-p_0),\ldots,p_0+\lambda_n(p_n-p_0)}).
\end{equation*}
}
\end{enumerate}
\end{proposition}
\begin{proof}
First, if $p_0$ minimizes
\begin{equation*}
x\mapsto \sum_{i=1}^n \gamma(p_i-x),
\end{equation*}
then $p_0$ minimizes also
\begin{equation*}
x\mapsto \sum_{i=0}^n \gamma(p_i-x),
\end{equation*}
since, for any $x\in X$, we have
\begin{equation*}
\sum_{i=0}^n \gamma(p_i-p_0)=\sum_{i=1}^n \gamma(p_i-p_0)\leq \sum_{i=1}^n \gamma(p_i-x)\leq \sum_{i=0}^n \gamma(p_i-x).
\end{equation*}
Second, suppose that $p_0\in\ft(\setn{p_1,\ldots,p_n})$, and let $q_i=p_0+\lambda_i(p_i-p_0)$ for each $i\in\setn{1,\ldots,n}$. Without loss of generality, assume that $p_0=0$. Evidently, $0\in\ft(\setn{p_1,\ldots,p_n})$ if and only if $0\in\ft(\setn{\lambda p_1,\ldots,\lambda p_n})$ for any $\lambda\in\RR_{++}$. In other words, the Fermat--Torricelli locus is compatible with scaling. Thus, we may assume that each $\lambda_i\leq 1$ by making the original configuration of the given points $p_i$ sufficiently large. Then, for $x\in X$,
\begin{align*}
\sum_{i=1}^n \gamma(q_i-0)&=\sum_{i=1}^n\lr{\gamma(p_i-0)-\gamma(p_i-q_i)}\\
&\leq \sum_{i=1}^n\lr{\gamma(p_i-x)-\gamma(p_i-q_i)}\\
&\leq \sum_{i=1}^n\gamma(q_i-x),
\end{align*}
\dah $0\in\ft(\setn{q_1,\ldots,q_n})$.\qed
\end{proof}

\begin{proposition}[{see \cite[Corollary~3.1]{MartiniSwWe2002}}]
Let $\setn{p_0,\ldots,p_n}\subseteq X$, $n \geq 2$. If $p_0\in \ft(\setn{p_1,\ldots,p_n})\setminus \setn{p_1,\ldots,p_n}$, then also $p_0\in\ft(\setn{p_0,p_1,\ldots,p_{n-1}})$.
\end{proposition}
\begin{proof}
Using Theorem~\ref{thm:fermat-torricelli}, we have
\begin{align*}
p_0\in \ft(\setn{p_1,\ldots,p_n})&\Longleftrightarrow \begin{cases}\fall i\in\setn{1,\ldots,n}\ex \phi_i\in X^\ast, \gamma^\circ(\phi_i)=1,\\\skpr{\phi_i}{p_i-p_0}=\gamma(p_i-p_0): \sum_{i=1}^n\phi_i =0\end{cases}\\
&\Longrightarrow \begin{cases}\fall i\in\setn{1,\ldots,n-1}\ex \phi_i\in X^\ast, \gamma^\circ(\phi_i)=1,\\\skpr{\phi_i}{p_i-p_0}=\gamma(p_i-p_0): \gamma^\circ\lr{-\sum_{i=1}^{n-1}\phi_i}= 1\end{cases}\\
&\Longrightarrow p_0\in\ft(\setn{p_0,p_1,\ldots,p_{n-1}}).
\end{align*}
\qed
\end{proof}

Next we generalize a theorem on strictly convex norms to gauges. As for norms, we say that a gauge $\gamma$ on $X$ is \emph{strictly convex} if the unit ball $\mcg{\gamma}{0}{1}$ is strictly convex or, equivalently, if no line segment of positive length is a subset of the sphere $\msg{\gamma}{0}{1}$.
\begin{proposition}[{see \cite[Theorem~3.3]{MartiniSwWe2002}}]\label{prop-strict_convex}
If $(X,\gamma)$ is a generalized Min\-kow\-ski space with strictly convex gauge, then $\ft(P)$ is a singleton for every non-collinear subset $P=\setn{p_1,\ldots,p_n}\subseteq X$.
\end{proposition}
\begin{proof}
Suppose $x$, $y\in\ft(P)$, $x\neq y$. By convexity of $\ft(P)$ (Proposition~\ref{prop:ft_compact_convex}), we have $[x,y]\subseteq \ft(P)$. Since $P$ is finite, we may assume $x,y\notin P$. Thus there exist, by Theorem~\ref{thm:fermat-torricelli}, $\gamma$-norming functionals $\phi_i$ of $p_i-x$ for each $p_i\in P$ such that $\sum_{i=1}^n\phi_i=0$. We have
\begin{align*}
\sum_{i=1}^n \gamma(p_i-x)&=\sum_{i=1}^n \skpr{\phi_i}{p_i-x}\\
&=\sum_{i=1}^n \skpr{\phi_i}{p_i-y}+\sum_{i=1}^n\skpr{\phi_i}{y-x}\\
&\leq\sum_{i=1}^n \gamma(p_i-y)\\
&=\sum_{i=1}^n\gamma(p_i-x).
\end{align*}
It follows that $\skpr{\phi_i}{p_i-y}=\gamma(p_i-y)$ for each $p_i\in P$ or, in other words, that $\phi_i$ is also a $\gamma$-norming functional of $p_i-y$. Since $P$ is non-collinear, there is $p_i\in P$ such that $x$, $y$, and $p_i$ are not collinear. Hence, $\frac{p_i-x}{\gamma(p_i-x)}$ and $\frac{p_i-y}{\gamma(p_i-y)}$ are different unit vectors with a common $\gamma$-norming functional, \dah $\clseg{\frac{p_i-x}{\gamma(p_i-x)}}{\frac{p_i-y}{\gamma(p_i-y)}}$ is a segment on the unit sphere, contradiction.\qed
\end{proof}

\begin{lemma}[{see \cite[Lemma~4.1]{MartiniSwWe2002}}]
Let $P=\setn{p_1,\ldots,p_n}$ be a finite subset of a generalized Minkowski space $(X,\gamma)$ such that $p_n\in\ft(P)$. Then $p_n$ is an exposed point of the polytope $\co(P\cap \ft(P))$, and
\begin{equation*}
\setcond{\frac{p_n-q}{\gamma(p_n-q)}}{q \in \ft(P)\cap P, q\neq p_n}
\end{equation*}  
is contained in an exposed face of the unit ball.
\end{lemma}
\begin{proof}
By Theorem~\ref{thm:fermat-torricelli}, there exist $\gamma$-norming functionals $\phi_i$ of $p_i-p_n$ for each $i\in\setn{1,\ldots,n-1}$ such that $\gamma^\circ\lr{-\sum_{i=1}^{n-1}\phi_i}\leq 1$. Then, for any $q\in  P\cap\ft(P)$ with $q\neq p_n$, we have
\begin{align*}
\sum_{i=1}^n \gamma(p_i-q)&=\sum_{i=1}^{n-1} \gamma(p_i-p_n)\\
&=\sum_{i=1}^{n-1}\skpr{\phi_i}{p_i-p_n}\\
&=\sum_{i=1}^{n-1}\skpr{\phi_i}{p_i-q}-\sum_{i=1}^{n-1}\skpr{\phi_i}{p_n-q}\\
&\leq \sum_{i=1}^{n-1} \gamma^\circ(\phi_i)\gamma(p_i-q)+ \gamma^\circ\lr{-\sum_{i=1}^{n-1} \phi_i}\gamma(p_n-q)\\
&\leq \sum_{i=1}^{n-1}\gamma(p_i-q)+ \gamma(p_n-q)\\
&=\sum_{i=1}^n \gamma(p_i-q).
\end{align*}
It follows, that $\phi\defeq -\sum_{i=1}^{n-1}\phi_i$ is a $\gamma$-norming functional of $p_n-q$. In other words, $\frac{p_n-q}{\gamma(p_n-q)}\in \phi^{-1}(1)\cap \mc{0}{1}$, which is an exposed face of the unit ball $\mc{0}{1}$. Furthermore, $\skpr{\phi}{p_n}=\gamma(p_n-q)+\skpr{\phi}{q}$ for all $q\in  P\cap\ft(P)$, $q\neq p_n$, \dah $\phi$ strictly separates $\setn{p_n}$ and $(P\cap\ft(P))\setminus\setn{p_n}$. Thus $p_n$ is a exposed point of $\co(P\cap \ft(P))$.\qed
\end{proof}
\begin{lemma}[{see \cite[Corollary~3.3]{MartiniSwWe2002}}]
Let a finite subset $P$ of a generalized Minkowski space $(X,\gamma)$ be split into disjoint non-empty sets $P_1,\ldots,P_m$ such that $\bigcap_{j=1}^m \ft(P_j) \neq \emptyset$. Then $\ft(P)=\bigcap_{j=1}^m \ft(P_j)$.
\end{lemma}
The last claim is an immediate consequence of the following simple fact.
\begin{lemma}
Let $f_1,\ldots,f_m: S \to \RR$, $m \geq 1$, be functions on the same set $S$. If there exists a common minimum point $\bar{x}$ of $f_1,\ldots,f_m$, then $x \in S$ is a minimum point of $f=\sum_{j=1}^m f_j$ if and only if $x$ is a common minimum point of $f_1,\ldots,f_m$.
\end{lemma}
\begin{proof}
For all $x \in S$,
\begin{equation*}
f(x)=\sum_{j=1}^m f_j(x) \geq \sum_{j=1}^m f_j(\bar{x}) =f(\bar{x})
\end{equation*}
with equality if and only if $f_j(x)=f_j(\bar{x})$, $j \in \setn{1,\ldots,m}$. This yields the claim.\qed
\end{proof}
 
\subsection{Differences from the Norm Setting and the Role of Metrically Defined Segments}
\label{chap:differences}
Menger \cite{Menger1928} considers a kind of betweenness relation in metric spaces which forms the basis for the more modern notion of $d$-segments in normed spaces, see \cite[Chapter II]{BoltyanskiMaSo1997}. In a normed space $(X,\mnorm{\cdot})$, the \emph{$d$-segment} between two points $x,y \in X$ is defined as
\begin{equation*}
[x,y]_d=\setcond{z\in X}{\mnorm{x-z}+\mnorm{z-y}=\mnorm{x-y}}.
\end{equation*}
Clearly, $[x,y]_d=[y,x]_d$. The triangle inequality shows that $\ft(\setn{x,y})=[x,y]_d$ (see \cite[p.~290]{MartiniSwWe2002}), and there are several other connections of the Fermat--Torricelli problem with the notion of $d$-segments (see \cite[Proposition~3.3, Corollaries~3.3~and~3.5, Theorem~4.1]{MartiniSwWe2002}).

When working with generalized Minkowski spaces $(X,\gamma)$, we give the analogous definition
\begin{equation*}
[x,y]_\gamma=\setcond{z\in X}{\gamma(x-z)+\gamma(z-y)=\gamma(x-y)},
\end{equation*}
but now we cannot expect $[x,y]_\gamma=[y,x]_\gamma$ in general. 

\begin{lemma}\label{lem-segments}
Let $x,y$ be two points of a generalized Minkowski space $(X,\gamma)$. 
\begin{enumerate}[label={(\alph*)},leftmargin=0.3cm,align=left]
\item{$[x,y]_\gamma$ is closed and convex.\label{dseg_closed_convex}}
\item{$[x,y] \subseteq [x,y]_\gamma$.\label{seg_in_dseg}}
\item{$[x,y]_\gamma=\bigcup_{0\leq \lambda \leq 1}\big(x-\ms{0}{\lambda\gamma(x-y)}\big)\cap \big(y+\ms{0}{(1-\lambda)\gamma(x-y)}\big)$.\label{dseg_construction}}
\end{enumerate}
\end{lemma}
\begin{proof}
\ref{dseg_closed_convex} Convexity of $[x,y]_\gamma$ is a consequence of positive homogeneity and of the triangle inequality. And $[x,y]_\gamma$ is closed, because $\gamma$ is continuous.

\ref{seg_in_dseg} One easily checks that $x,y \in [x,y]_\gamma$. Then $[x,y] \subseteq [x,y]_\gamma$, because $[x,y]_\gamma$ is convex.

\ref{dseg_construction} Observe that
\begin{align*}
[x,y]_\gamma&=\setcond{z\in X}{\gamma(x-z)+\gamma(z-y)=\gamma(x-y)}\\
&=\bigcup_{0\leq \lambda \leq 1} \setcond{z\in X}{\begin{matrix}\gamma(x-z)=\lambda\gamma(x-y),\\\lambda\gamma(x-y)+\gamma(z-y)=\gamma(x-y)\end{matrix}}\\
&=\bigcup_{0\leq \lambda \leq 1} \setcond{z\in X}{\begin{matrix}\gamma(x-z)=\lambda\gamma(x-y),\\\gamma(z-y)=(1-\lambda)\gamma(x-y)\end{matrix}}\\
&=\bigcup_{0\leq \lambda \leq 1} \setcond{z\in X}{\begin{matrix}x-z\in \ms{0}{\lambda\gamma(x-y)},\\z-y\in \ms{0}{(1-\lambda)\gamma(x-y)}\end{matrix}}\\
&=\bigcup_{0\leq \lambda \leq 1} \setcond{z\in X}{\begin{matrix}z\in x-\ms{0}{\lambda\gamma(x-y)},\\z\in y+\ms{0}{(1-\lambda)\gamma(x-y)}\end{matrix}}\\
&=\bigcup_{0\leq \lambda \leq 1}(x-\ms{0}{\lambda\gamma(x-y)})\cap (y+\ms{0}{(1-\lambda)\gamma(x-y)}).
\end{align*}
\qed
\end{proof}

Using the above mentioned equation $\ft(\setn{x,y})=[x,y]_d$ from \cite[p.~290]{MartiniSwWe2002}, we know now that
\begin{equation}
\setn{x,y} \subseteq [x,y] \subseteq [x,y]_\gamma = [x,y]_d = \ft(\setn{x,y})\label{eq:norm_situation}
\end{equation}
for arbitrary $x,y \in X$, provided that $\gamma$ is a norm. The situation is different if $\gamma$ is not a norm.

\begin{proposition}\label{prop:asymmetric}
Let $(X,\gamma)$ be a generalized Minkowski space such that $\gamma$ is not a norm. Then there exists $x_0 \in X \setminus \setn{0}$ such that $\ft(\setn{x_0,0})=\setn{0}$.
\end{proposition}

\begin{proof}
By compactness of $\ms{0}{1}$, there exists $x_0 \in \ms{0}{1}$ such that
\begin{equation*}
\gamma(-x_0)=\max\setcond{\gamma(-x)}{x \in \ms{0}{1}}=\max\setcond{\frac{\gamma(-x)}{\gamma(x)}}{x \in X \setminus \setn{0}}.
\end{equation*}
Since $\gamma$ is not symmetric,
\begin{equation}
\gamma(-x_0)>\gamma(x_0)=1.\label{eq1}
\end{equation}
Moreover,
\begin{equation}
\gamma(x) \geq \frac{\gamma(-x)}{\gamma(-x_0)} \quad\text{ for all }x \in X.\label{eq2}
\end{equation}

$\ft(\setn{x_0,0})$ consists of all minimizers of $f:X \to \RR, f(x)=\gamma(x_0-x)+\gamma(0-x)$. In order to show that $\ft(\setn{x_0,0})=\setn{0}$, it is enough to prove that $f(x)>f(0)$ for all $x \in X \setminus \setn{0}$. For arbitrary $x \neq 0$, we estimate
\setlength{\mywidth}{\widthof{$\stackrel{\text{(triangle inequality)}}{\geq}$}}
\begin{align*}
f(x) &\gleich \gamma(x_0-x)+\gamma(-x)\\
&\groessergleich[\text{\eqref{eq2}}]\frac{1}{\gamma(-x_0)}\gamma(-x_0+x)+\gamma(-x)\\
&\groessergleich[\text{(triangle inequality)}]\frac{1}{\gamma(-x_0)}\Big(\gamma(-x_0)-\gamma(-x)\Big)+\gamma(-x)\\
&\gleich 1+ \lr{1-\frac{1}{\gamma(-x_0)}}\gamma(-x)\\
&\groesser[\text{\eqref{eq1}}] 1\\
&\gleich[\text{\eqref{eq1}}] \gamma(x_0)\\
&\gleich f(0),
\end{align*}
and the proof is complete.\qed
\end{proof}

\begin{remark}
The above proof together with a dilatation argument shows the following: If $(X,\gamma)$ is a generalized Minkowski space and if $x_0 \in X \setminus\setn{0}$ satisfies
\begin{equation*}
\frac{\gamma(-x_0)}{\gamma(x_0)}= \max\setcond{\frac{\gamma(-x)}{\gamma(x)}}{ x \in X \setminus \setn{0}}>1,
\end{equation*}
then $\ft(\setn{x_0,0})=\setn{0}$.
\end{remark}

We obtain serveral characterizations of norms among arbitrary gauges.
\begin{corollary}\label{cor-norm_equivalent}
Let $(X,\gamma)$ be a generalized Minkowski space. The following are equivalent.
\begin{enumerate}[label={(\alph*)},leftmargin=0.3cm,align=left]
\item{$\gamma$ is a norm.\label{norm}}
\item{For any two distinct points $x,y \in X$, $\abs{\ft(\setn{x,y})}>1$.}
\item{For any two distinct points $x,y \in X$, $\abs{\ft(\setn{x,y})}=\infty$.}
\item{For any two distinct points $x,y \in X$, $\ft(\setn{x,y}) \setminus \setn{x} \neq \emptyset$.}
\item{For any two distinct points $x,y \in X$, $\ft(\setn{x,y}) \setminus \setn{x,y} \neq \emptyset$.}
\item{For any two distinct points $x,y \in X$, $x \in \ft(\setn{x,y})$.}
\item{For any two distinct points $x,y \in X$, $\setn{x,y} \subseteq \ft(\setn{x,y})$.}
\item{For any two distinct points $x,y \in X$, $[x,y] \subseteq \ft(\setn{x,y})$.}
\item{For any two distinct points $x,y \in X$, $\ri([x,y]) \cap \ft(\setn{x,y}) \neq \emptyset$.}
\item{For any two distinct points $x,y \in X$, $[x,y]_\gamma \subseteq \ft(\setn{x,y})$.}
\item{For any two distinct points $x,y \in X$, $\ri([x,y]_\gamma) \cap \ft(\setn{x,y}) \neq \emptyset$.\label{ri_gamma}}
\item{For any two distinct points $x,y \in X$, $[x,y]_\gamma = \ft(\setn{x,y})$.\label{dseg=ft}}
\end{enumerate}
\end{corollary}

\begin{proof}
We know from \eqref{eq:norm_situation} that \ref{norm} implies the other conditions. Proposition~\ref{prop:asymmetric} shows that each of the other conditions implies \ref{norm}. However, we give details for the implication \ref{ri_gamma}$\Rightarrow$\ref{norm}, since it is less obvious.

We assume that \ref{ri_gamma} is satisfied, whereas \ref{norm} fails. Then Proposition~\ref{prop:asymmetric} provides $x_0 \in X \setminus \setn{0}$ such that $\ft(\setn{x_0,0})=\setn{0}$, and \ref{ri_gamma} yields $0 \in\ri([x_0,0]_\gamma)$. Since $[x_0,0] \subseteq [x_0,0]_\gamma$ (see Lemma~\ref{lem-segments}\ref{seg_in_dseg}), $0$ must be an inner point of the set $[x_0,0]_\gamma \cap \aff\setn{x_0,0}$ relative to the natural topology of $\aff\setn{x_0,0}=\lin\setn{x_0}$. Hence there exists $\eps > 0$ such that $-\eps x_0 \in [x_0,0]_\gamma \cap \aff\setn{x_0,0} \subseteq [x_0,0]_\gamma$. The inclusion $-\eps x_0 \in [x_0,0]_\gamma$ yields
\begin{equation*}
\gamma(x_0-(-\eps x_0))+\gamma(-\eps x_0-0)=\gamma(x_0-0).
\end{equation*}
We obtain the contradiction $(1+\eps)\gamma(x_0)+\gamma(-\eps x_0)=\gamma(x_0)$, and the proof is complete.\qed
\end{proof}

\begin{example}\label{ex:negative}
Take $X=\RR^2$ and
\begin{equation*}
\gamma(\xi_1,\xi_2)=\max\setn{-\frac{1}{2}\xi_1, \xi_1+\xi_2, \xi_1-\xi_2, \frac{1}{2}\xi_1+\xi_2, \frac{1}{2}\xi_1-\xi_2}.
\end{equation*}

For $x=(-2,2)$ and $y=(-2,-2)$, we have $\ft(\setn{x,y})=\setn{(0,0)}$. This can be shown by elementary calculations, and Figure~\ref{fig:ft_subset_gamma-segment} illustrates level curves of the corresponding objective function.
\begin{figure}[h!]
\begin{center}
\begin{tikzpicture}

\begin{axis}[%
width=4in,
height=3.15483870968in,
unbounded coords=jump,
scale only axis,
xmin=-4.53762781186094,
xmax=4.33762781186094,
ymin=-3.5,
ymax=3.5,
axis x line*=bottom,
axis y line*=left
]

\addplot [draw=black,forget plot] table[row sep=crcr]{
-0.095 0.1\\
-0.1 0.095\\
-0.1 -0.095\\
-0.095 -0.1\\
0.095 -0.1\\
0.1 -0.095\\
0.1 0.095\\
0.095 0.1\\
-0.095 0.1\\
NaN NaN\\
};

\addplot [draw=black,forget plot] table[row sep=crcr]{
-0.595 0.6\\
-0.6 0.595\\
-0.6 -0.595\\
-0.595 -0.6\\
0.595 -0.6\\
0.6 -0.595\\
0.6 0.595\\
0.595 0.6\\
-0.595 0.6\\
NaN NaN\\
};

\addplot [draw=black,forget plot] table[row sep=crcr]{
-1.1 1.0975\\
-1.1 -1.0975\\
-1.095 -1.1\\
1.095 -1.1\\
1.1 -1.095\\
1.1 1.095\\
1.095 1.1\\
-1.095 1.1\\
-1.1 1.0975\\
NaN NaN\\
};

\addplot [draw=black,forget plot] table[row sep=crcr]{
-1.6 1.6\\
-1.6 -1.6\\
1.6 -1.6\\
1.6 1.6\\
-1.6 1.6\\
NaN NaN\\
};

\addplot [draw=black,forget plot] table[row sep=crcr]{
2.095 2.1\\
-1.9 2.1\\
-2 2.05\\
-2.05 2\\
-2.05 -2\\
-2 -2.05\\
-1.9 -2.1\\
2.095 -2.1\\
2.1 -2.095\\
2.1 2.095\\
2.095 2.1\\
NaN NaN\\
};

\addplot [draw=black,forget plot] table[row sep=crcr]{
2.595 2.6\\
-1.4 2.6\\
-2 2.3\\
-2.3 2\\
-2.3 -2\\
-2 -2.3\\
-1.4 -2.6\\
2.595 -2.6\\
2.6 -2.595\\
2.6 2.595\\
2.595 2.6\\
NaN NaN\\
};

\addplot [draw=black,forget plot] table[row sep=crcr]{
3.095 3.1\\
-0.9 3.1\\
-2 2.55\\
-2.55 2\\
-2.55 -2\\
-2 -2.55\\
-0.9 -3.1\\
3.095 -3.1\\
3.1 -3.095\\
3.1 3.095\\
3.095 3.1\\
NaN NaN\\
};

\addplot [draw=black,forget plot] table[row sep=crcr]{
-0.6 3.5\\
-2 2.8\\
-2.8 2\\
-2.8 -2\\
-2 -2.8\\
-0.6 -3.5\\
NaN NaN\\
};

\addplot [draw=black,forget plot] table[row sep=crcr]{
-1.1 3.5\\
-2 3.05\\
-3.05 2\\
-3.05 -2\\
-2 -3.05\\
-1.1 -3.5\\
NaN NaN\\
};

\addplot [draw=black,forget plot] table[row sep=crcr]{
-1.6 3.5\\
-2 3.3\\
-3.3 2\\
-3.3 -2\\
-2 -3.3\\
-1.6 -3.5\\
NaN NaN\\
};

\addplot [draw=black,forget plot,line width=1.8pt,mark=*,mark size={1pt}] table[row sep=crcr]{
-2 -2\\
-2 2\\
};

\node[below right] at (axis cs:-2,2) {$x$};
\node[above right] at (axis cs:-2,-2) {$y$};
\end{axis}
\end{tikzpicture}%
\end{center}
\caption{The Fermat--Torricelli locus of two points does not necessarily belong to the intersection of the respective $d$-segments.}\label{fig:ft_subset_gamma-segment}
\end{figure}
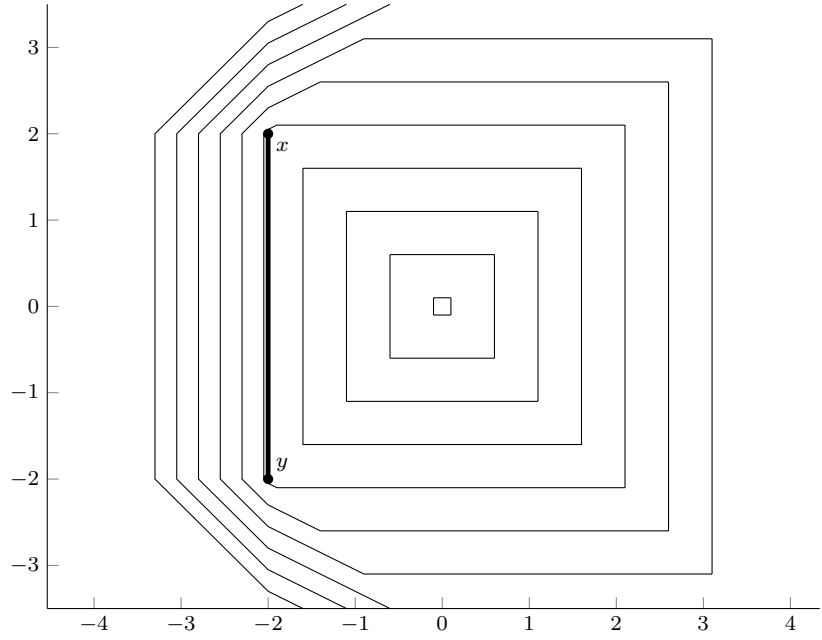
Moreover, we obtain $[x,y]_\gamma=[y,x]_\gamma=[x,y]$ (bold line in Figure~\ref{fig:ft_subset_gamma-segment}), which can be shown with the help of Lemma~\ref{lem-segments}\ref{dseg_construction} (see Figure~\ref{fig:construction_gamma-segment} for an illustration). 
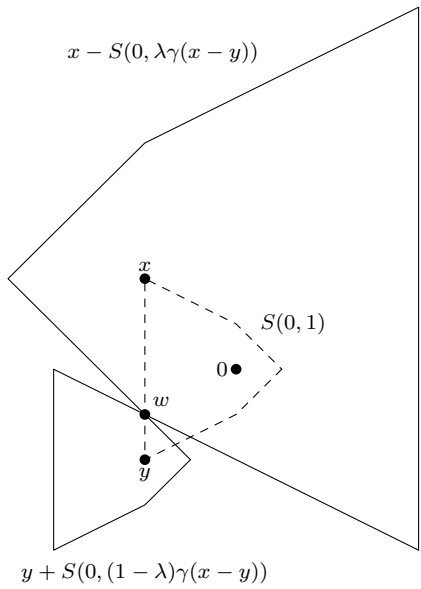
\begin{figure}[h!]
\begin{center}
\begin{tikzpicture}[line cap=round,line join=round,>=triangle 45,x=0.6cm,y=0.6cm]
\draw[color=black,dashed] (1,0) -- (0,-1) -- (-2,-2) -- (-2,2) -- (0,1) -- cycle;
\draw[color=black] (-5,2) -- (-2,5) -- (4,8) -- (4,-4) -- (-2,-1) -- cycle;
\draw[color=black] (-1,-2) -- (-2,-3) -- (-4,-4) -- (-4,0) -- (-2,-1) -- cycle;
\fill[color=black] (-2,-2) circle (2pt) node[below] {$y$};
\fill[color=black] (0,0) circle (2pt) node[left] {$0$};
\fill[color=black] (-2,2) circle (2pt) node[above] {$x$};
\fill[color=black] (-2,-1) circle (2pt) node[above right] {$w$};
\draw[color=black] (1.25,1) node{$\ms{0}{1}$};
\draw[color=black] (-2,-4.5) node{$y+\ms{0}{(1-\lambda)\gamma(x-y)}$};
\draw[color=black] (-1.6,7) node{$x-\ms{0}{\lambda\gamma(x-y)}$};
\end{tikzpicture}
\end{center}
\caption{Constructing $[x,y]_\gamma$: $w$ is a point of $[x,y]_\gamma$.}\label{fig:construction_gamma-segment}
\end{figure}

In our example, the Fermat--Torricelli locus $\ft(\setn{x,y})$ has the following two properties, that are known to be impossible in classical normed spaces.
\begin{enumerate}[label={(\Roman*)},leftmargin=0.3cm,align=left,series=failure]
\item{There is a finite set $P$ in a two-dimensional generalized Minkowski space $(X,\gamma)$ such that $\co(P)\cap \ft(P) = \emptyset$ (confer \cite[Theorem~3.4]{MartiniSwWe2002} for the classical setting).}
\item{We have $\ft(\setn{x,y})\centernot\subseteq [x,y]_{\gamma}$ (whereas $\ft(\setn{x,y})=[x,y]_d$ in the classical setting, see \cite[p.~290]{MartiniSwWe2002} and also Corollary~\ref{cor-norm_equivalent}\ref{dseg=ft}).}
\end{enumerate}
\end{example}

\begin{example}
Strict convexity of norms can be characterized by means of the Fermat--Torricelli locus: the norm of a Minkowski space $(X,\mnorm{\cdot})$ is strictly convex if and only if $\ft(P)$ is a singleton for every non-collinear set $P=\setn{p_1,\ldots,p_n} \subseteq X$, see \cite[Theorem~3.3]{MartiniSwWe2002}. For generalized Minkowski spaces, this condition remains necessary for strict convexity of gauges (see Proposition~\ref{prop-strict_convex}), but the sufficiency fails.

\begin{enumerate}[label={(\Roman*)},leftmargin=0.3cm,align=left,failure]
\item{There is a generalized Minkowski space $(X,\gamma)$ such that, for any non-collinear subset $P=\setn{p_1,\ldots,p_n} \subseteq X$, $\abs{\ft(P)}=1$, but $\gamma$ is not strictly convex.}
\end{enumerate}
Take $X=\RR^2$ and
\begin{equation*}
\gamma(\xi_1,\xi_2)=\begin{cases}\abss{\xi_1}+\abss{\xi_2},& \xi_1,\xi_2 \in \RR_+,\\\sqrt{\xi_1^2+\xi_2^2},&\text{otherwise.}\end{cases}
\end{equation*}
Assume that there is a non-collinear set $P=\setn{p_1,\ldots,p_n} \subseteq X$ such that $\abs{\ft(P)}\neq 1$. Then, by convexity of $\ft(P)$, there is a point $x\in\ft(P)\setminus P$, and, by Theorem~\ref{thm:fermat-torricelli} and Proposition~\ref{prop:ft_locus_cone_intersection}, there are $\gamma$-norming functionals $\phi_i$ of $p_i-x$ (for $i \in \setn{1,\ldots,n}$) such that $\sum_{i=1}^n \phi_i=0$ and $\ft(P)$ is the intersection of the cones $C(p_i,\phi_i)$ apexed at $p_i$ and generated by exposed faces of $p_i-\mc{0}{1}$. Note that, after identifying $X^\ast$ with $\RR^2$, every $\phi_i$ belongs to
\begin{align*}
\msg{\gamma^\circ}{0}{1}&=\setcond{(\cos \alpha, \sin \alpha)}{\frac{\piup}{2} < \alpha < 2\piup}\\
&\qquad\cup \setcond{(1,\beta)}{0 \leq \beta \leq 1}\cup \setcond{(\beta,1)}{0 \leq \beta \leq 1},
\end{align*}
and $C(p_i,\phi_i)$ is a ray if $\phi_i \neq (1,1)$ and represents an angle of size $\frac{\piup}{2}$ if $\phi_i=(1,1)$. Suppose that $\phi_i \neq (1,1)$ for $i\in\setn{1,\ldots,k}$ and $\phi_i=(1,1)$ for $k < i \leq n$. Since $\abs{\ft(P)}>1$, all rays $C(p_i,\phi_i)$, $i\in\setn{1,\ldots,k}$, are parallel. Hence $\phi_1,\ldots,\phi_k \in \setn{\pm(\cos \alpha_0, \sin \alpha_0)}$ with fixed $\alpha_0 \in \lr{\frac{\piup}{2},\piup}$ (Case 1: rays with negative slope) or $\phi_1,\ldots,\phi_k \in \setcond{(1,\beta)}{0 \leq \beta < 1} \cup \setn{(-1,0)}$ (Case 2: horizontal rays) or $\phi_1,\ldots,\phi_k \in \setcond{(\beta,1)}{0 \leq \beta < 1} \cup \setn{(0,-1)}$ (Case 3: vertical rays) or $\phi_1=\ldots=\phi_k=(\cos \alpha_0, \sin \alpha_0)$ with fixed $\alpha_0 \in \lr{\piup,\frac{3\piup}{2}}$ (Case 4: rays with positive slope). Since $P$ is not collinear, we obtain $k < n$ (otherwise $P$ would be contained in the straight line passing through $x$ and parallel to the rays). The equation $\sum_{i=1}^n \phi_i=0$ yields
\begin{equation*}
(\xi_1^\ast,\xi_2^\ast)\defeq \sum_{i=1}^k \phi_i= -\sum_{i=k+1}^n \phi_i= (n-k)(-1,-1).
\end{equation*}
This is impossible in Case 1, because then $(\xi_1^\ast,\xi_2^\ast)=l(\cos \alpha_0, \sin \alpha_0)$, with $l \in \ZZ$, is either zero or its coordinates have different signs. In Case 2 (Case 3), we obtain a contradiction, since then $\xi_2^\ast \geq 0$ ($\xi_1^\ast \geq 0$). Finally,
Case 4 gives $(\xi_1^\ast,\xi_2^\ast)=k(\cos \alpha_0, \sin \alpha_0)$, the equality $\xi_1^\ast=-(n-k)=\xi_2^\ast$ implies $\alpha_0=\frac{5\piup}{4}$, and we obtain a contradiction from $-(n-k)=\xi_1^\ast=-k\frac{\sqrt{2}}{2}$, since $\sqrt{2}$ is irrational.
\end{example}

The last example shows that a characterization of strict convexity of norms does not extend to arbitrary gauges. We shall give a characterization in terms of segments. It is based on a generalization of another statement on normed spaces.

\begin{lemma}[see {\cite[Proposition~1]{MartiniSwWe2001}}]\label{lem-triangle=}
For all points $x,y\in X \setminus\setn{0}$ of a generalized Minkowski space $(X,\gamma)$, $\gamma (x+y)= \gamma(x)+ \gamma(y)$ if and only if $\clseg{\frac{x}{\gamma(x)}}{\frac{y}{\gamma(y)}}\subseteq \ms{0}{1}$.
\end{lemma}
\begin{proof}
We can assume that $x+y \neq 0$, because otherwise $y=-x$, and the claim is trivial.

If $\gamma(x+y)=\gamma(x)+\gamma(y)$, then
\begin{equation*}
\frac{x+y}{\gamma(x+y)}=\frac{\gamma(x)}{\gamma(x+y)}\frac{x}{\gamma(x)}+\frac{\gamma(y)}{\gamma(x+y)}\frac{y}{\gamma(y)},
\end{equation*}
\dah the unit vector $\frac{x+y}{\gamma(x+y)}$ is a convex combination of the unit vectors $\frac{x}{\gamma(x)}$ and $\frac{y}{\gamma(y)}$, hence $\clseg{\frac{x}{\gamma(x)}}{\frac{y}{\gamma(y)}}\subseteq \ms{0}{1}$. 

Conversely, if $\clseg{\frac{x}{\gamma(x)}}{\frac{y}{\gamma(y)}}\subseteq \ms{0}{1}$, we have
\begin{equation*}
\frac{x+y}{\gamma(x)+\gamma(y)}=\frac{\gamma(x)}{\gamma(x)+\gamma(y)}\frac{x}{\gamma(x)}+\frac{\gamma(y)}{\gamma(x)+\gamma(y)}\frac{y}{\gamma(y)},
\end{equation*}
\dah $\frac{x+y}{\gamma(x)+\gamma(y)}$ is a point of the segment $\clseg{\frac{x}{\gamma(x)}}{\frac{y}{\gamma(y)}}$; hence $\frac{x+y}{\gamma(x)+\gamma(y)}$ is a unit vector or, equivalently, $\gamma(x+y)=\gamma(x)+\gamma(y)$.\qed
\end{proof}

\begin{proposition}
A generalized Minkowski space $(X,\gamma)$ has a strictly convex gauge if and only if $[x,y]_\gamma=[x,y]$ for all $x,y \in X$.
\end{proposition}

\begin{proof}
We prove that there exist $x,y \in X$ such that $[x,y]_\gamma \neq [x,y]$ if and only if $\gamma$ is not strictly convex. By Lemmas~\ref{lem-segments}\ref{seg_in_dseg} and \ref{lem-triangle=}, the condition $[x,y]_\gamma \neq [x,y]$ is equivalent to $[x,y]_\gamma \centernot\subseteq [x,y]$, which is in turn equivalent to
\begin{equation}
\label{eq-segment}
\exists\; z \in X \setminus [x,y]: \; \clseg{\frac{x-z}{\gamma(x-z)}}{\frac{z-y}{\gamma(z-y)}} \subseteq \ms{0}{1}.
\end{equation}
\eqref{eq-segment} shows that $\gamma$ is not strictly convex, because 
$\clseg{\frac{x-z}{\gamma(x-z)}}{\frac{z-y}{\gamma(z-y)}}$ is not degenerate (\dah a point), since this would imply 
\begin{equation*}
z=\frac{\gamma(z-y)}{\gamma(x-z)+\gamma(z-y)}x+\frac{\gamma(x-z)}{\gamma(x-z)+\gamma(z-y)}y \in [x,y].
\end{equation*}
Conversely, if $\gamma$ is not strictly convex, there exist linearly independent $z_1,z_2 \in X$ such that $[z_1,z_2] \subseteq \ms{0}{1}$. This yields \eqref{eq-segment} for $x=z_1+z_2$, $y=0$, and $z=z_1$.\qed
\end{proof}

\subsection{Boundary Structure of Sublevel Sets of the Objective Function}
\label{chap:boundary_structure}
Figure~\ref{fig:ft_subset_gamma-segment} illustrates sublevel sets 
\begin{equation*}
f_{\leq \alpha}= \setcond{x \in X}{f(x) \leq \alpha}
\end{equation*}
of the Fermat--Torricelli objective function $f=\gamma(p_1-\,\cdot)+\gamma(p_2-\,\cdot)$ for two points $p_1,p_2 \in X=\RR^2$. Every extreme point $x_0$ of the sublevel set $f_{\leq \alpha}$ is of the form $x_0=p_i+\lambda w$, where $i \in \setn{1,2}$, $\lambda \in\RR_+$, and $w$ is a extreme point of $-\mc{0}{1}$. This turns out to be a particular case of a more general phenomenon.

Recall that a \emph{$k$-face}, $0 \leq k \leq d$, of a closed convex set $K \subseteq X=\RR^d$ is a subset $F \subseteq K$ of affine dimension $k$ such that, whenever a relative interior point of a line segment $[x,y] \subseteq K$ meets $F$, then $[x,y] \subseteq F$ \cite[pp.~18,~62]{Schneider1993}. In particular, every $k$-dimensional intersection of $K$ with a supporting hyperplane is a $k$-face. These are called \emph{exposed} $k$-faces, cf. Definition~\ref{def:exposedness_cones}\ref{exposedness} and \cite[p.~63]{Schneider1993}.
For example, $x$ is an extreme point of $K$ if and only if $\setn{x}$ is a $0$-face, and $x$ is an exposed point of $K$ if and only if $\setn{x}$ is an exposed $0$-face. 

According to \cite[Theorem~2.1.2]{Schneider1993}, every $x \in K$ belongs to the relative interior $\ri(F_x)$ of a unique face $F_x$ of $K$. The point $x$ is called \emph{$k$-extreme} if $\dim(F_x) \leq k$. Clearly, $x$ is \emph{extreme} if and only if it is $0$-extreme. The \emph{$k$-skeleton} of $K$ is the set
\begin{equation*}
\ext_k(K)=\setcond{x \in K}{x \text{ is } k\text{-extreme}},
\end{equation*}
see \cite[p.~65]{Schneider1993}.

\begin{lemma}\label{lem-extreme}
Every $(d-1)$-face of a closed convex set $K \subseteq X= \RR^d$, $d\geq 2$, is exposed. In particular,
\begin{equation*}
\ext_{d-2}(K)= \bd(K) \setminus \bigcup\setcond{\ri(F)}{F \text{ is an exposed } (d-1)\text{-face of } K}.
\end{equation*}
\end{lemma} 

\begin{proof}
Let $F$ be a $(d-1)$-face of $K$. If $F$ were not exposed, then the hyperplane $H_F$ spanned by $F$ would not support $K$. Then $F=H_F \cap K$ would contain interior points of $K$, and $F$ were not a face, a contradiction. Hence $F$ is exposed.

By \cite[Theorem~2.1.2]{Schneider1993}, $\bd(K)$ is the disjoint union
\begin{equation*}
\bd(K)= \bigcup\setcond{\ri(F)}{F \text{ is a } k\text{-face of } K \text{ with } k \leq d-1}.
\end{equation*}
Therefore,
\begin{align*}
\ext_{d-2}(K) &= \bigcup\setcond{\ri(F)}{F \text{ is a } k\text{-face of } K \text{ with } k \leq d-2}\\
&=\bd(K) \setminus \bigcup\setcond{\ri(F)}{F \text{ is a } (d-1)\text{-face of } K}\\
&=\bd(K) \setminus \bigcup\setcond{\ri(F)}{F \text{ is an exposed } (d-1)\text{-face of } K}.
\end{align*}
\qed
\end{proof}

\begin{proposition}\label{prop-sublevel}
Let $X=\RR^d$, $d\geq 2$, and let $p_1,\ldots,p_n \in X$, $n \geq 1$. Furthermore, let $\gamma_1 \ldots,\gamma_n$ be gauges on $X$, $f=\sum_{i=1}^n \gamma_i(p_i-\cdot)$ be the corresponding Fermat--Torricelli objective function, and let $\alpha \in \RR$. Then every point $x_0 \in \ext_{d-2}(f_{\leq \alpha})$ can be expressed as $x_0= p_i+\lambda w$ with suitable $i \in \setn{1,\ldots,n}$, a real number $\lambda \in [0,\alpha]$, and a point $w \in \ext_{d-2}(-\mcg{\gamma_i}{0}{1})$.
\end{proposition}

\begin{proof}
The claim is trivial if $x_0=p_i$ for some $i \in \setn{1,\ldots,n}$, because then $x_0=p_i+0w$ for every $w \in\ext_{d-2}(-\mcg{\gamma_i}{0}{1})$. Hence we can assume that $x_0 \neq p_i$, $i\in\setn{1,\ldots,n}$. Putting $\lambda_i=\gamma_i(p_i-x_0)\in\RR_{++}$, we obtain
\begin{gather}
x_0=p_i+\lambda_i w_i \;\text{ with }\; \lambda_i\in\RR_{++},\; w_i \in -\msg{\gamma_i}{0}{1} \;\text{ for }\; i \in \setn{1,\ldots,n},
\label{sl1}\\
\sum_{i=1}^n \lambda_i = \alpha.\nonumber
\end{gather}

The proof is complete when we have shown that $w_i \in \ext_{d-2}(-\mcg{\gamma_i}{0}{1})$ for some $i$. Suppose that this is not the case. Then, for every $i \in \setn{1,\ldots,n}$, we have $w_i \in \ri(F_i)$ for some exposed $(d-1)$-face $F_i$ of $-\mcg{\gamma_i}{0}{1}$, according to Lemma~\ref{lem-extreme}. Denoting the corresponding supporting hyperplane by $H_i$, we have
\begin{equation}
w_i \in \ri(F_i)=\ri(H_i \cap (-\msg{\gamma_i}{0}{1})).\label{sl2}
\end{equation}
Since $F_i$ is of dimension $d-1$, the set
\begin{equation*}
K_i=\setcond{\lambda w}{w \in \ri(H_i \cap (-\msg{\gamma_i}{0}{1})), \lambda \in\RR_{++}}
\end{equation*}
is open in $X$. The restricted function $\gamma_i(-\,\cdot)|_{K_i}$ is linear, because
\begin{equation*}
\gamma_i(-\lambda w) \in \gamma_i(\lambda((-H_i) \cap \msg{\gamma_i}{0}{1})) \subseteq \lambda \gamma_i(\msg{\gamma_i}{0}{1})=\setn{\lambda}
\end{equation*}
for all $w \in \ri(H_i \cap (-\msg{\gamma_i}{0}{1}))$ and $\lambda\in\RR_{++}$. Formulas \eqref{sl1} and \eqref{sl2} and the linearity of $\gamma_i(-\,\cdot)|_{K_i}$ show that
$p_i+K_i$ is an open neighbourhood of $x_0$ and the restricted function $f_i|_{p_i+K_i}=\gamma_i(p_i-\,\cdot)|_{p_i+K_i}$ is affine.

Now it follows that $K=(p_1+K_1) \cap \ldots \cap (p_n+K_n)$ is an open neighbourhood of $x_0$ and that $f|_K=(\sum_{i=1}^n f_i|_{p_i+K_i})|_K$ is an affine function. However, boundary points from a level set of an affine function are never $(d-2)$-extreme. Therefore, $x_0 \notin \ext_{d-2}(f_{\leq \alpha})$. This contradiction completes the proof.\qed
\end{proof}

\begin{remark}
For the case $d=2$, Proposition~\ref{prop-sublevel} says that every extreme point $x_0$ of $f_{\leq \alpha}$ is of the form $x_0=p_i+\lambda w$ for some $i \in \setn{1,\ldots,n}$, $\lambda \in [0,\alpha]$, and an extreme point $w$ of $-\mcg{\gamma_i}{0}{1}$. This fails in general for spaces $X=\RR^d$ of dimensions $d>2$, as it is illustrated by the following example.

Let $X=\RR^3$, $n=3$, $p_1=(1,0,-1)$, $p_2=(0,-1,1)$, $p_3=(-1,1,0)$, let $\gamma_1=\gamma_2=\gamma_3=\mnorm{\cdot}_1$ be the norm given by $\mnorm{(\xi_1,\xi_2,\xi_3)}_1=\abs{\xi_1}+\abs{\xi_2}+\abs{\xi_3}$, and let $\alpha=6$. Then $f_{\leq 6}=\setn{(0,0,0)}$, because $f(0,0,0)=6$ and, for arbitrary $(\xi_1,\xi_2,\xi_3) \neq (0,0,0)$ with $\max\setn{\abs{\xi_1},\abs{\xi_2},\abs{\xi_3}} \leq 1$,
\begin{align*}
f(\xi_1,\xi_2,\xi_3) &=
(\abs{\xi_1-1}+\abs{\xi_2}+\abs{\xi_3+1})+(\abs{\xi_1}+\abs{\xi_2+1}+\abs{\xi_3-1})\\
& \qquad+(\abs{\xi_1+1}+\abs{\xi_2-1}+\abs{\xi_3})\\
&= (\abs{\xi_1-1}+\abs{\xi_1}+\abs{\xi_1+1})+(\abs{\xi_2-1}+\abs{\xi_2}+\abs{\xi_2+1})\\
& \qquad+(\abs{\xi_3-1}+\abs{\xi_3}+\abs{\xi_3+1})\\
&= ((1-\abs{\xi_1})+\abs{\xi_1}+(1+\abs{\xi_1}))+((1-\abs{\xi_2})+\abs{\xi_2}+(1+\abs{\xi_2}))\\
& \qquad+((1-\abs{\xi_3})+\abs{\xi_3}+(1+\abs{\xi_3}))\\
&= 6+3\mnorm{(\xi_1,\xi_2,\xi_3)}_1\\
&> 6.
\end{align*}
Hence $x_0=(0,0,0)$ is an extreme point of $f_{\leq 6}$. But one easily checks that $x_0$ does not admit a representation $x_0=p_i+\lambda w$ with $i \in \setn{1,2,3}$, $\lambda \in\RR_+$, and an extreme point $w$ of $-\mcg{\gamma_i}{0}{1}$.
\end{remark}

\section{Finitely Many Convex Sets in Generalized Minkowski Spaces}
\label{chap:sets_minkowski}
We give a generalization of Theorem~\ref{thm:fermat-torricelli} regarding Problem \eqref{eq:fermat-torricelli}. Namely, we replace the given point set $\setn{p_1,\ldots,p_n}$ by a collection of convex sets $K_1,\ldots,K_n$. Distance measurement is then provided by the so-called distance functions related to convex sets and gauges.
\begin{definition}The \emph{distance function of $K\subseteq X$ with respect to the gauge $\gamma:X\to \RR$} is defined as $\dist_\gamma(\cdot,K):X\to\cRR$ by
\begin{equation*}
\dist_\gamma(x,K)=\inf\setcond{\gamma(y-x)}{y\in K}.
\end{equation*}
The \emph{metric projection onto $K$ with respect to $\gamma$} is the set-valued operator $\Proj_\gamma(\cdot,K):X\rightrightarrows X$,
\begin{equation*}
\Proj_\gamma(x,K)=\setcond{y\in K}{\gamma(y-x)=\dist_\gamma(x,K)}.
\end{equation*}
If the dependence on $\gamma$ is clear from the context, we may omit it from the notation.
\end{definition}
Now let $K_1,\ldots,K_n\subseteq X$ be non-empty closed convex sets, and let $\gamma_1,\ldots,\gamma_n$ be gauges on $X$. Consider the convex optimization problem
\begin{equation}
\inf_{x\in X}\sum_{i=1}^n \dist_{\gamma_i}(x,K_i).\label{eq:fermat-torricelli_general}
\end{equation}

We start the discussion of \eqref{eq:fermat-torricelli_general} with an analogue of Proposition~\ref{prop:ft_compact_convex}. As above, we call the set of all minimizers of \eqref{eq:fermat-torricelli_general} the \emph{Fermat--Torricelli locus} of \eqref{eq:fermat-torricelli_general}. A particular emphasis is on the case where all the sets $K_i$ are affine flats (\dah affine subspaces) in $X$.
\begin{proposition}\label{prop-locus_general}
\begin{enumerate}[label={(\alph*)},leftmargin=0.3cm,align=left]
\item{The Fermat--Torricelli locus of \eqref{eq:fermat-torricelli_general} is closed and convex.\label{locus_closed_convex}}
\item{If one of the sets $K_i$, $i\in\setn{1,\ldots,n}$, is bounded, then the Fermat--Torricelli locus of \eqref{eq:fermat-torricelli_general} is bounded and non-empty.\label{one_bounded}}
\item{If all the sets $K_i$, $i\in\setn{1,\ldots,n}$, are affine flats, then the Fermat--Torricelli locus of \eqref{eq:fermat-torricelli_general} is non-empty. Moreover, 
it is a Minkowski sum of a closed, bounded, convex set and a linear space that may be degenerated to $\setn{0}$.\label{all_flats}}
\end{enumerate}
\end{proposition}

\begin{proof}
Claim \ref{locus_closed_convex} holds, since the solution set of \eqref{eq:fermat-torricelli_general} is a sublevel set of a bounded below, convex, and continuous function $f=\sum_{i=1}^n \dist_{\gamma_i}(\cdot,K_i)$.

Statement \ref{one_bounded} can be proved in a similar way as \cite[Proposition~4.1(i)]{MordukhovichNa2011b}: Suppose that $K_{i_0}$ is bounded. If $\alpha=\inf_{x \in X} f(x)$ then all
minimizers of $f$ are contained in the sublevel set $H=\setcond{x \in X}{\dist_{\gamma_{i_0}}(x,K_{i_0}) \leq \alpha}$ of $\dist_{\gamma_{i_0}}(\cdot,K_{i_0})$. Problem \eqref{eq:fermat-torricelli_general} has a bounded and non-empty set of minimizers, since they are obtained by minimizing the continuous function $f$ over the non-empty compact set $H$.

For \ref{all_flats}, we proceed by induction on $d=\dim(X)$. Statement \ref{all_flats} is a simple consequence of \ref{one_bounded} if $d=1$. 

Now suppose that $d \geq 2$ and, contrary to our claim, all sets $K_i$ are affine flats, but either $f$ has no minimizer in $X$ (Case (A)) or the Fermat--Torricelli locus is non-empty and cannot be represented as it is claimed under \ref{all_flats} (Case (B)). Note that the locus must be unbounded in the latter case.

We fix a norm $\mnorm{\cdot}$ on $X$. There exists a sequence $(x_k)_{k=1}^\infty \subseteq X$ such that 
\begin{equation*}
\lim_{k \to \infty} f(x_k)= \inf_{x \in X} f(x)= \alpha \quad\text{ and }\quad \lim_{k \to \infty} \mnorm{x_k}= \infty.
\end{equation*}
To see that in Case (A), pick 
$(x_k)_{k=1}^\infty$ such that $\lim_{k \to \infty} f(x_k)= \alpha$. Then necessarily $\lim_{k \to \infty} \mnorm{x_k}= \infty$, since otherwise $(x_k)_{k=1}^\infty$ had a bounded subsequence converging to a minimizer of $f$. In Case (B), we pick an arbitrary unbounded sequence $(x_k)_{k=1}^\infty$ in the Fermat--Torricelli locus.

We can assume that $\lim_{k \to \infty} \frac{x_k}{\mnorm{x_k}}=x_0 \in \msg{\mnorm{\cdot}}{0}{1}$, since $\left(\frac{x_k}{\mnorm{x_k}}\right)_{k=1}^\infty$ is contained in the compact set $\msg{\mnorm{\cdot}}{0}{1}$. The property $\lim_{k \to \infty} f(x_k)=\alpha$ implies boundedness of the set $\setcond{f(x_k)}{k \geq 1}$ and in turn of each of the sets $\setcond{\dist_{\gamma_i}(x_k,K_i)}{k \geq 1}$ for $i\in\setn{1,\ldots,n}$. Denoting the linear subspace associated to $K_i$ by $V_i$ (\dah $K_i=V_i+x$ for every $x \in K_i$), we conclude that $\setcond{\dist_{\gamma_i}(x_k,V_i)}{k \geq 1}$ is bounded from above by some $C_i \in\RR_{++}$. By the linearity of $V_i$, we get
\begin{equation*}
\dist_{\gamma_i}\lr{\frac{x_k}{\mnorm{x_k}}, V_i}=\frac{1}{\mnorm{x_k}}\dist_{\gamma_i}(x_k, V_i) \leq \frac{C_i}{\mnorm{x_k}}
\end{equation*}
for $k \geq 1$. Letting $k \to \infty$, we obtain $\dist_{\gamma_i}(x_0,V_i)=0$, \dah $x_0 \in V_i$, and 
\begin{equation}
K_i= K_i+\lambda x_0 \quad\text{ for }\quad \lambda \in \RR,\, i\in\setn{1,\ldots,n}.\label{eq-invariance}
\end{equation}

We express $X$ as a direct sum 
\begin{equation}
X=\bar{X} \oplus \lin\setn{x_0}\label{eq-sum_representation}
\end{equation}
with a linear subspace $\bar{X} \subseteq X$ of dimension $\dim(\bar{X})=\dim(X)-1=d-1$. For $\bar{x} \in \bar{X}$ and $i \in \setn{1,\ldots,n}$, let
\begin{equation*}
\bar{\gamma}_i(\bar{x})= \inf\setcond{\gamma_i(\bar{x}+\lambda x_0)}{\lambda \in \RR}.
\end{equation*}
Then every $\bar{\gamma}_i$ turns out to be a gauge on $\bar{X}$ (see \cite[Proposition~(3.1)]{AlegreFe2007}). Indeed, $\bar{\gamma}_i(\bar{x})=0$ is equivalent to $\dist_{\tilde{\gamma}_i}(\bar{x},\lin\setn{x_0})=0$, \dah to $\bar{x} \in \lin\setn{x_0}$. By \eqref{eq-sum_representation}, this gives $\bar{x}=0$, and part \ref{linefree} of Definition~\ref{def-gauge} is verified. Parts \ref{homogeneous} and \ref{subadditive} follow easily from the respective properties of $\gamma_i$.

From \eqref{eq-invariance} and \eqref{eq-sum_representation}, we obtain
\begin{equation*}
K_i= \bar{K}_i+ \lin\setn{x_0}, \quad\text{ where }\quad \bar{K}_i=K_i \cap \bar{X},
\end{equation*}
for $i \in \setn{1,\ldots,n}$. We define an optimization problem on $\bar{X}$ by
\begin{equation}
\inf_{\bar{x} \in \bar{X}} \bar{f}(\bar{x}), \quad\text{ where }\quad \bar{f}(\bar{x})= \sum_{i=1}^n \dist_{\bar{\gamma}_i}(\bar{x},\bar{K}_i).\label{eq-new_problem}
\end{equation}
For $\bar{x} \in \bar{X}$ and $i\in\setn{1,\ldots,n}$,
\begin{align*}
\dist_{\bar{\gamma}_i}(\bar{x},\bar{K}_i) 
&= \inf\setcond{\bar{\gamma}_i(\bar{y}-\bar{x})}{\bar{y} \in \bar{K}_i}\\
&= \inf\setcond{\inf\setcond{\gamma_i(\bar{y}-\bar{x}+\lambda x_0)}{\lambda \in \RR}}{\bar{y} \in \bar{K}_i}\\
&= \inf\setcond{\gamma_i((\bar{y}+\lambda{x_0})-\bar{x})}{\bar{y} \in \bar{K}_i,\lambda \in \RR}\\
&= \inf\setcond{\gamma_i(y-\bar{x})}{y \in K_i}\\
&= \dist_{\gamma_i}(\bar{x},K_i).
\end{align*}
This gives the identity $\bar{f}(\bar{x})=f(\bar{x})$ for all $\bar{x} \in \bar{X}$. Moreover, \eqref{eq-invariance} shows that $f(x)=f(x+\lambda x_0)$ for all $x \in X$, $\lambda \in \RR$. So
\begin{equation*}
\bar{f}(\bar{x})=f(\bar{x}+\lambda x_0) \quad\text{ for }\quad \bar{x} \in \bar{X},\, \lambda \in \RR.
\end{equation*}
Now we see that $\bar{f}$ and $f$ attain the same values. Moreover, the Fermat--Torricelli loci $\bar{F}$ and $F$ of \eqref{eq-new_problem} and \eqref{eq:fermat-torricelli_general}, respectively, are related by
\begin{equation*}
F=\bar{F}+\lin\setn{x_0}.
\end{equation*}
However, the induction hypothesis tells us that $\bar{F}=\bar{K}+\bar{V}$, where $\bar{K} \subseteq \bar{X}$ is non-empty, closed, bounded, and convex and $\bar{V}$ is a linear subspace of $\bar{X}$. Then $F=\bar{K}+V$, where $V=\bar{V}+\lin\setn{x_0}$ is a subspace of $X$, and the proof is complete.\qed
\end{proof}

\begin{example}
\begin{enumerate}[label={(\alph*)},leftmargin=0.3cm,align=left]
\item{The Fermat--Torricelli locus of \eqref{eq:fermat-torricelli_general} can be empty. For example, consider $n=2$ sets 
\begin{equation*}
K_1=\setcond{(\xi_1,0)}{\xi_1 \in \RR}\text{ and }K_2=\setcond{(\xi_1,\xi_2) \in \RR}{\xi_1 \in\RR_{++}, \xi_2 \geq \frac{1}{\xi_1}}
\end{equation*}
in $\RR^2$ equipped with arbitrary gauges $\gamma_1,\gamma_2$.}
\item{The Fermat--Torricelli locus of \eqref{eq:fermat-torricelli_general} can be unbounded, even if all sets $K_i$ are affine flats. For example, this appears if $n=2$ and $K_1$ and $K_2$ are two parallel
straight lines in $\RR^2$ equipped with arbitrary gauges $\gamma_1,\gamma_2$. If $\gamma_1=\gamma_2$ is a norm, then the Fermat--Torricelli locus is the complete strip between $K_1$ and $K_2$, which shows that the Fermat--Torricelli locus is not necessarily an affine flat if all sets $K_i$ are affine flats.}
\end{enumerate}
\end{example}

In order to give an optimality condition for \eqref{eq:fermat-torricelli_general}, we compute the conjugate and the subdifferential of the function $\dist_{\gamma}(\cdot,K)$, where $\gamma$ is a gauge on $X$ and $K$ is a non-empty closed convex set.
Using the \emph{indicator function} $\delta(\cdot,K):X\to \cRR$ of $K$ with
\begin{equation*}
\delta(x,K)=\begin{cases}0,&x\in K,\\+\infty,&x\notin K,\end{cases}
\end{equation*}
we have $\dist_\gamma(x,K)=\inf\setcond{\tilde{\gamma}(x-y)+\delta(y,K)}{y\in X}= (\tilde{\gamma}\boxx \delta(\cdot,K))(x)$. (Here $(f \boxx g)(x)\defeq\inf\setcond{f(x-y)+g(y)}{y \in X}$ is the convolution of the functions $f,g: X \to \cRR$; see, \zB \cite[p.~167]{BauschkeCo2011} and \cite[p.~43]{Zalinescu2002}.) By \cite[Proposition~13.21(i)]{BauschkeCo2011},
\begin{equation}
\dist_\gamma(\cdot,K)^\ast=(\tilde{\gamma}\boxx \delta(\cdot,K))^\ast=\tilde{\gamma}^\ast+\delta(\cdot,K)^\ast= \delta(\cdot,-\mc{0}{1}^\circ)+h(\cdot,K).\label{eq-dist_star}
\end{equation}
The subdifferential can be computed via \cite[Theorem~2.4.2(iii)]{Zalinescu2002}. Namely,
\begin{align}
&\norel\p \dist_\gamma(\cdot,K)(x)\nonumber\\
&=\setcond{\phi \in X^\ast}{\dist_\gamma(\cdot,K)^\ast(\phi)+\dist_\gamma(x,K)=\skpr{\phi}{x}} \nonumber\\
&=\setcond{\phi \in X^\ast}{\delta(\phi,-\mc{0}{1}^\circ)+h(\phi,K)+\dist_\gamma(x,K)=\skpr{\phi}{x}} \nonumber\\
&=\setcond{\phi \in X^\ast}{\tilde{\gamma}^\circ(\phi)\leq 1, h(\phi,K)+\dist_\gamma(x,K)=\skpr{\phi}{x}} \label{eq-d_dist_4}\\
&= -\mc{0}{1}^\circ \cap \setcond{\phi \in X^\ast}{\dist_\gamma(x,K)=\inf_{y\in K}\skpr{\phi}{x-y}}. \label{eq-d_dist_2}
\end{align}
Now we are able to formulate an optimality condition for Problem $\eqref{eq:fermat-torricelli_general}$.
\begin{theorem}\label{thm:fermat-torricelli_general}
Let $K_1,\ldots,K_n\subseteq X$ be non-empty closed convex sets, and let $\gamma_1,\ldots,\gamma_n$ be gauges on $X$. Then $\bar{x}\in X$ is a minimizer of the function $f=\sum_{i=1}^n\dist_{\gamma_i}(\cdot,K_i):X\to\RR$ if and only if there exist $\phi_1,\ldots,\phi_n\in X^\ast$ with $\gamma_i^\circ(-\phi_i)\leq 1$ and $\dist_{\gamma_i}(\bar{x},K_i)=\inf_{y\in K_i}\skpr{\phi_i}{\bar{x}-y}$ such that $\sum_{i=1}^n \phi_i=0$.
\end{theorem}
\begin{proof}
Both conditions are equivalent to $0 \in \p f(\bar{x})$.\qed
\end{proof}
Note that there is no need for a positively weighted version of Problem~\eqref{eq:fermat-torricelli_general} and Theorem~\ref{thm:fermat-torricelli_general}, since $w\gamma$ is a gauge if $\gamma$ is a gauge and $w\in \RR_{++}$. The theorem also covers the cases $n=0$ (where $f \equiv 0$ and the optimality condition is empty) and $n=1$. The sets $K_i$ are not necessarily different nor disjoint. A special case of Theorem~\ref{thm:fermat-torricelli_general} is given in \cite{Plastria1992b}, where every gauge $\gamma_i$ is a so-called \emph{skewed norm}, \dah its dual unit ball $\mcg{\gamma_i^\circ}{0}{1}$ admits a center of symmetry.

We come to a restricted version of Problem \eqref{eq:fermat-torricelli_general}, which is also called \emph{generalized Heron problem} in the literature (see \cite{Ghandehari1997}, \cite{MordukhovichNaSa2012b}, \cite{MordukhovichNaSa2012a}, and, strongly related, \cite{GhandehariGo2000}). Let $K_0,K_1,\ldots,K_n\subseteq X$ be non-empty closed convex sets, and let $\gamma_1,\ldots,\gamma_n$ be gauges on $X$. Consider
\begin{equation}
\inf_{x\in K_0}\;\sum_{i=1}^n \dist_{\gamma_i}(x,K_i).\label{eq:heron}
\end{equation}
The existence of an optimal solution for \eqref{eq:heron} can be shown as in \cite[Proposition~3.1]{MordukhovichNaSa2012b} if one of the sets $K_i$, $i\in\setn{0,\ldots,n}$, is bounded. This optimization problem can be rewritten as
\begin{align*}
\inf_{x\in X}&\;\sum_{i=1}^n \dist_{\gamma_i}(x,K_i)+\delta(x,K_0).
\end{align*}
For deducing an optimality condition, we note that the second set in \eqref{eq-d_dist_2} coincides in the case $x \in K$ with
\begin{align}
&\norel\setcond{\phi \in X^\ast}{\dist_\gamma(x,K)=\inf_{y\in K}\skpr{\phi}{x-y}}\nonumber\\
&=\setcond{\phi \in X^\ast}{0=\inf_{y\in K}\skpr{\phi}{x-y}} \nonumber \\
&=\setcond{\phi \in X^\ast}{0=\sup_{y\in K}\skpr{\phi}{y-x}} \nonumber \\
&=\setcond{\phi \in X^\ast}{0\geq\skpr{\phi}{y-x}\fall y\in K} \nonumber \\
&=\nor(x,K), \label{eq-d_dist_3}
\end{align}
the \emph{normal cone} of $K$ at $x$ (see \cite[p.~70]{Schneider1993}). In particular, by \eqref{eq-d_dist_2} we have
\begin{equation*}
\p \dist_\gamma(\cdot,K)(x)= -\mc{0}{1}^\circ \cap \nor(x,K) \;\text{ for }\; x \in K
\end{equation*}
(see \cite[Example~16.49]{BauschkeCo2011} for the case that $X$ is Euclidean). This formula is a finite-di\-men\-sional special case of formula (17) in \cite{MordukhovichNa2011b}.

\begin{theorem}[{see \cite{MordukhovichNaSa2012b}}]
Let $K_0,\ldots,K_n\subseteq X$ be non-empty closed convex sets, and let $\gamma_1,\ldots,\gamma_n$ be gauges on $X$. A point $\bar{x}\in K_0$ is a minimizer of the function $f=\sum_{i=1}^n\dist_{\gamma_i}(\cdot,K_i)+\delta(\cdot,K_0):X\to\cRR$ if and only if there exist functionals $\phi_0\in\nor(\bar{x},K_0)$ and $\phi_i \in X^\ast$ satisfying $\gamma_i^\circ(-\phi_i)\leq 1$ and $\dist_{\gamma_i}(\bar{x},K_i)=\inf_{y\in K_i}\skpr{\phi_i}{\bar{x}-y}$, $i \in \setn{1,\ldots,n}$, such that $\sum_{i=0}^n \phi_i=0$.
\end{theorem}
\begin{proof}
Both conditions are equivalent to $0 \in \p f(\bar{x})$, see \eqref{eq-d_dist_2}, \eqref{eq-d_dist_3}, and \cite[Corollary~16.39]{BauschkeCo2011} for the additivity of the subdifferential.\qed
\end{proof}

Now we come to interesting particular results.

\begin{proposition}
For every closed convex cone $C\subseteq X$ apexed at $0$ in a generalized Minkowski space $(X,\gamma)$, we have $\dist(\cdot,C)=h(\cdot,C^\circ\cap (-\mc{0}{1}^\circ))$.
\end{proposition}
\begin{proof}
By \eqref{eq-dist_star},
\begin{align*}
\delta(\cdot,C^\circ\cap (-\mc{0}{1}^\circ))&=\delta(\cdot,-\mc{0}{1}^\circ) + \delta(\cdot,C^\circ)\\
&=\delta(\cdot,-\mc{0}{1}^\circ) + h(\cdot,C)\\
&=\dist(\cdot,C)^\ast.
\end{align*}
The assertion now follows from taking conjugates and applying the Fenchel--Moreau theorem \cite[Theorem~13.32]{BauschkeCo2011}.\qed
\end{proof}

\begin{proposition}
Let $K$ be an affine flat in a generalized Minkowski space $(X,\gamma)$, i.e., $K=r+U$ with a linear subspace $U$ of $X$ and $r \in X$. Then, for all $x \in X$,
\begin{equation*}
\p \dist(\cdot,K)(x) =U^\perp\cap (-\mc{0}{1}^\circ) \cap \setcond{\phi \in X^\ast}{\dist(x,K)=\skpr{\phi}{x-r}},
\end{equation*}
where $U^\perp\defeq\setcond{\phi\in X^\ast}{\skpr{\phi}{x}=0 \fall x\in U}$.
\end{proposition}

\begin{proof}
For $\phi \in X^\ast$,
\begin{equation*}
h(\phi,K)=\left\{\begin{matrix}\skpr{\phi}{r},&\phi\in U^\perp,\\+\infty,&\text{otherwise}\end{matrix}\right\}=\begin{cases}\skpr{\phi}{r},&\phi\in (K-r)^\perp,\\+\infty,&\text{otherwise.}\end{cases}
\end{equation*}
Therefore, by \eqref{eq-d_dist_4},
\begin{align*}
\p \dist(\cdot,K)(x)&=\setcond{\phi \in X^\ast}{\tilde{\gamma}^\circ(\phi)\leq 1, h(\phi,K)+\dist(x,K)=\skpr{\phi}{x}} \\
&=\setcond{\phi \in X^\ast}{\gamma^\circ(-\phi)\leq 1,\phi \in U^\perp, \dist(x,K)=\skpr{\phi}{x-r}}\\
&=-\mc{0}{1}^\circ \cap U^\perp \cap \setcond{\phi \in X^\ast}{\dist(x,K)=\skpr{\phi}{x-r}}.
\end{align*}
\qed
\end{proof}

\begin{remark}
One might expect that Proposition~\ref{prop-sublevel} (at least if $d=2$) can be extended to the case of non-empty compact, convex sets $K_1,\ldots,K_n\subseteq X= \RR^2$ instead of $p_1,\ldots,p_n$ as in Theorem~\ref{thm:fermat-torricelli_general} in the following sense: Every extreme point $x_0$ of $f_{\leq \alpha}$ can be expressed as $x_0= p+\lambda w$ with an extreme point $p$ of $K_i$ for suitable $i \in \setn{1,\ldots,n}$, a real number $\lambda \in\RR_+$, and an extreme point $w$ of $-\mcg{\gamma_i}{0}{1}$. But this is not the case, as the following example shows.

Consider $n=4$,
\begin{align*}
K_1&=[(-4,-1),(4,-1)],&K_2&=[(-4,1),(4,1)],\\
K_3&=[(-1,-4),(-1,4)],&K_4&=[(1,-4),(1,4)],
\end{align*}
let $\gamma_1=\gamma_2=\gamma_3=\gamma_4= \mnorm{\cdot}_\infty$ be the maximum norm, and let $\alpha=4$. Then the extreme points of $f_{\leq 4}=[-1,1]^2$ do not admit the above representation.
\end{remark}

\section{Finitely Many Convex Sets in Euclidean Spaces}
\label{chap:sets_euclidean}
If $\gamma_1=\ldots=\gamma_n= \gamma$ and $(X,\gamma)=(\EE^d,\mnorm{\cdot})$ is the $d$-dimensional Euclidean space with scalar product $\skpr{\cdot}{\cdot}$, we may identify $(X^\ast, \gamma^\circ)\cong (X,\gamma)$. Then the metric projection $\Proj(\cdot,K): \EE^d \rightrightarrows \EE^d$ onto a non-empty closed convex set is singleton-valued and can be understood as a map into $\EE^d$. Theorem~\ref{thm:fermat-torricelli_general} reduces to
\begin{theorem}[{see \cite[Theorem~4.2]{MordukhovichNa2011b}}]\label{thm:hilbert_minisum_single}
Let $K_1,\ldots,K_n\subseteq \EE^d$ be non-empty closed convex sets, and let $f: \EE^d\to\RR$, $f(x)=\sum_{i=1}^n \dist(x,K_i)$. Then the following are equivalent for every $\bar{x} \in \EE^d$.
\begin{enumerate}[label={(\alph*)},leftmargin=0.3cm,align=left]
\item{The point $\bar{x}$ is a minimizer of $f$.\label{hilbertcase_a}}
\item{There exist points $z_1,\ldots,z_n\in \EE^d$ satisfying the relations $\mnorm{z_i}\leq 1$ and $\dist(\bar{x},K_i)=\inf_{y\in K_i}\skpr{z_i}{\bar{x}-y}$ such that $\sum_{i=1}^n z_i=0$.\label{hilbertcase_b}}
\item{We have
\begin{equation*}
\sum_{\substack{i=1,\ldots,n\\\bar{x}\notin K_i}} \frac{\Proj(\bar{x},K_i)-\bar{x}}{\mnorm{\Proj(\bar{x},K_i)-\bar{x}}}\in \sum_{\substack{i=1,\ldots,n\\\bar{x}\in K_i}} (\nor(\bar{x},K_i)\cap\mc{0}{1}).
\end{equation*}\label{hilbertcase_c}}
\end{enumerate}
\end{theorem}
\begin{proof}
The equivalence \ref{hilbertcase_a}$\Leftrightarrow$\ref{hilbertcase_b} is a direct consequence of Theorem~\ref{thm:fermat-torricelli_general}. 

For the implication \ref{hilbertcase_c}$\Rightarrow$\ref{hilbertcase_b}, set $z_i=\frac{\bar{x}-\Proj(\bar{x},K_i)}{\mnorm{\bar{x}-\Proj(\bar{x},K_i)}}$ if $\bar{x} \notin K_i$. Then \ref{hilbertcase_c} says that $-\sum_{\substack{i=1,\ldots,n\\\bar{x}\notin K_i}} z_i \in \sum_{\substack{i=1,\ldots,n\\\bar{x}\in K_i}} (\nor(\bar{x},K_i)\cap\mc{0}{1})$, and we can pick a point $z_i \in \nor(\bar{x},K_i)\cap\mc{0}{1}$ if $\bar{x} \in K_i$ such that $-\sum_{\substack{i=1,\ldots,n\\\bar{x}\notin K_i}} z_i = \sum_{\substack{i=1,\ldots,n\\\bar{x}\in K_i}} z_i$. This gives \ref{hilbertcase_b}.

Conversely, for \ref{hilbertcase_b}$\Rightarrow$\ref{hilbertcase_c} we proceed similarly. It suffices to show that, if $\bar{x}\notin K_i$, $\mnorm{z_i}\leq 1$, and $\dist(\bar{x},K_i)=\inf_{y\in K_i}\skpr{z_i}{\bar{x}-y}$, then $z_i=\frac{\bar{x}-\Proj(\bar{x},K_i)}{\mnorm{\bar{x}-\Proj(\bar{x},K_i)}}$. But
\begin{align*}
\mnorm{\bar{x}-\Proj(\bar{x},K_i)}&=\dist(\bar{x},K_i)\\
&=\inf_{y\in K_i}\skpr{z_i}{\bar{x}-y}\\
&\leq\skpr{z_i}{\bar{x}-\Proj(\bar{x},K_i)}\\
&\stackrel{\star}{\leq} \mnorm{z_i} \mnorm{\bar{x}-\Proj(\bar{x},K_i)}\\
&\leq \mnorm{\bar{x}-\Proj(\bar{x},K_i)}
\end{align*}
yields equality in the Cauchy--Schwarz inequality $\stackrel{\star}{\leq}$ and, hence, the desired representation for $z_i$.\qed
\end{proof}

Note that criterion \ref{hilbertcase_c} may involve a Minkowski sum over an empty set of indices. This sum is $\setn{0}$, the neutral element with respect to Minkowski summation. 

An alternative proof of \ref{hilbertcase_a}$\Leftrightarrow$\ref{hilbertcase_c} is possible with the aid of the subdifferential formula 
\begin{equation}
\p\dist(\cdot,K)(x)=\begin{cases}
\setn{\frac{x-\Proj(x,K)}{\mnorm{x-\Proj(x,K)}}},&x\notin K,\\[1ex]
\nor(x,K)\cap\mc{0}{1},&x\in K,
\end{cases}\label{eq-subdiff_euclidean}
\end{equation}
from \cite[Example~16.49]{BauschkeCo2011}. A straightforward modification allows the introduction of positive weights.

\begin{corollary}\label{cor-weights}
Let $K_1,\ldots,K_n\subseteq \EE^d$ be non-empty closed convex sets, let $w_1,\ldots,w_n \in \RR_{++}$ be positive weights. Then $\bar{x} \in \EE^d$ is a minimum point of $f:\EE^d \to\RR$, $f(x)=\sum_{i=1}^n w_i \dist(x,K_i)$ if and only if
\begin{equation*}
\label{eq-criterion}
\sum_{\substack{i=1,\ldots,n\\\bar{x}\notin K_i}} w_i\frac{\Proj(\bar{x},K_i)-\bar{x}}{\mnorm{\Proj(\bar{x},K_i)-\bar{x}}}\in \sum_{\substack{i=1,\ldots,n\\\bar{x}\in K_i}} (\nor(\bar{x},K_i)\cap\mc{0}{w_i}).
\end{equation*}
\end{corollary}

We give a formula for the \emph{directional derivatives} of the distance function. The directional derivative of a convex function $f:X\to\RR$ in the direction $y\in X$ is defined as 
\begin{equation*}
\frac{\p}{\p y^+}f:X\to \RR,\; \frac{\p}{\p y^+}f(x)\defeq \lim_{t\downarrow 0} \frac{f(x+ty)-f(x)}{t}.
\end{equation*}
\cite[Theorem~2.4.9]{Zalinescu2002} shows that
\begin{equation}
\frac{\p}{\p y^+}f(x)=\max\setcond{\skpr{\phi}{y}}{\phi \in \p f(x)}.\label{eq:directional_derivative_via_subdifferential}
\end{equation}

\begin{proposition}\label{prop-directional_derivatives}
Let $K$ be a non-empty closed convex set in $(\EE^d,\mnorm{\cdot})$, and let $x,y\in \EE^d$. Then
\begin{equation*}
\frac{\p}{\p y^+}\dist(\cdot,K)(x)=
\begin{cases}
\skpr{\frac{x-\Proj(x,K)}{\mnorm{x-\Proj(x,K)}}}{y},&x \notin K,\\
h(y,\nor(x,K)\cap\mc{0}{1}),& x \in K.\end{cases}
\end{equation*}
\end{proposition}
\begin{proof}
One combines \eqref{eq-subdiff_euclidean} with \eqref{eq:directional_derivative_via_subdifferential}.\qed
\end{proof}

Analogously like in the former section, we study now some particular consequences. In \cite{KupitzMaSp2013}, the Fermat--Torricelli problem is considered for affine flats in Euclidean space. If $K \subseteq \EE^d$ is an affine flat with $x \in K$, then $K=x+(K-x)$, where $K-x$ is a linear subspace of $\EE^d$, and 
\begin{equation*}
\nor(x,K)=(K-x)^\perp
\end{equation*}
consists of all vectors orthogonal to $K$. We generalize central results from \cite{KupitzMaSp2013}.

\begin{theorem}[see {\cite[Theorems 2.1, 3.1, and 4.2]{KupitzMaSp2013}}]\label{thm-Euclidean}
Let $K_1,\ldots,K_n$ be non-empty closed convex subsets of $(\EE^d,\mnorm{\cdot})$, let $\bar{x} \in \EE^d$, and let $f: \EE^d\to\RR$, $f(x)=\sum_{i=1}^n \dist(x,K_i)$.
\begin{enumerate}[label={(\alph*)},leftmargin=0.3cm,align=left]
\item{Floating case: if $\bar{x} \notin \bigcup_{i=1}^n K_i$, then $\bar{x}$ is a minimizer of $f$ if and only if 
\begin{equation*}
\sum_{i=1}^n \frac{\Proj(\bar{x},K_i)-\bar{x}}{\mnorm{\Proj(\bar{x},K_i)-\bar{x}}}=0.
\end{equation*}\label{floating}}
\item{Point absorbed case: if $\bar{x} \notin \bigcup_{i=1}^{n-1} K_i$ and $K_n=\setn{\bar{x}}$, then $\bar{x}$ is a minimizer of $f$ if and only if 
\begin{equation*}
\mnorm{\sum_{i=1}^{n-1} \frac{\Proj(\bar{x},K_i)-\bar{x}}{\mnorm{\Proj(\bar{x},K_i)-\bar{x}}}} \leq 1.
\end{equation*}\label{point_absorbed}}
\item{Flat absorbed case: if $\bar{x} \notin \bigcup_{i=1}^{n-1} K_i$ and $\bar{x} \in K_n$, where $K_n$ is an affine flat, then $\bar{x}$ is a minimizer of $f$ if and only if 
\begin{equation*}
\sum_{i=1}^{n-1} \frac{\Proj(\bar{x},K_i)-\bar{x}}{\mnorm{\Proj(\bar{x},K_i)-\bar{x}}} \text{ is orthogonal to } K_n
\end{equation*}
and
\begin{equation*}
\mnorm{\sum_{i=1}^{n-1} \frac{\Proj(\bar{x},K_i)-\bar{x}}{\mnorm{\Proj(\bar{x},K_i)-\bar{x}}}} \leq 1.
\end{equation*}
\label{flat_absorbed}}
\end{enumerate}
\end{theorem}

\begin{proof}
We use criterion \ref{hilbertcase_c} from Theorem~\ref{thm:hilbert_minisum_single}. It says
\begin{equation*}
\begin{array}{rcll}
\displaystyle\sum_{i=1}^n \frac{\Proj(\bar{x},K_i)-\bar{x}}{\mnorm{\Proj(\bar{x},K_i)-\bar{x}}} &\in&\setn{0}& \text{ in case \ref{floating},}\\
\displaystyle\sum_{i=1}^{n-1} \frac{\Proj(\bar{x},K_i)-\bar{x}}{\mnorm{\Proj(\bar{x},K_i)-\bar{x}}} &\in& \mc{0}{1}& \text{ in case \ref{point_absorbed},}\\
\displaystyle\sum_{i=1}^{n-1} \frac{\Proj(\bar{x},K_i)-\bar{x}}{\mnorm{\Proj(\bar{x},K_i)-\bar{x}}} &\in& (K_n-\bar{x})^\perp \cap \mc{0}{1}& \text{ in case \ref{flat_absorbed}.}
\end{array}
\end{equation*}
These imply the claims.\qed
\end{proof}

Finally we come to the so-called \emph{flat-point absorbed case} from \cite[Theorem~3.1(c)]{KupitzMaSp2013}. It concerns minimizers $\bar{x}$ of $f$ where $\setn{\bar{x}}$ represents one of the sets $K_i$ and $\bar{x}$ is contained in exactly one of the additional sets $K_i$, that, moreover, has to be an affine flat. The characterization from \cite{KupitzMaSp2013} is incorrect in so far as it is sufficient for minimality, but not necessary. Here we present a correction and slight generalization.

\begin{theorem}[see {\cite[Theorem~3.1(c)]{KupitzMaSp2013}}]
Let $K_1,\ldots,K_n$ be non-empty closed convex subsets of $(\EE^d,\mnorm{\cdot})$, where $K_{n-1}=\setn{\bar{x}} \subseteq K_n$, $\bar{x} \notin \bigcup_{i=1}^{n-2} K_i$, and $K_n$ is an affine flat of dimension $\dim(K_n) \in \setn{1,\ldots,d-1}$. Moreover, let
\begin{equation*}
v=\sum_{i=1}^{n-2} \frac{\Proj(\bar{x},K_i)-\bar{x}}{\mnorm{\Proj(\bar{x},K_i)-\bar{x}}},
\end{equation*}
let $\alpha \in \clseg{0}{\frac{\piup}{2}}$ be the angle between $v$ and $K_n$ if $v \neq 0$ and put $\alpha=0$ if $v=0$. Then the following are equivalent.
\begin{enumerate}[label={(\alph*)},leftmargin=0.3cm,align=left]
\item{$\bar{x}$ is a minimizer of $f: \EE^d\to\RR$, $f(x)=\sum_{i=1}^n \dist(x,K_i)$.\label{x_minimizer}}
\item{
\begin{equation*}
v \in \lr{(K_n-\bar{x})^\perp \cap \mc{0}{1}}+\mc{0}{1}.
\end{equation*}\label{v_orthogonality}}
\item{
\begin{equation*}
\mnorm{v}\leq \begin{cases} \frac{1}{\cos \alpha} &\text{if \,}0 \leq \alpha \leq \frac{\piup}{4},\\
2\sin \alpha &\text{if \,}\frac{\piup}{4} \leq \alpha \leq \frac{\piup}{2}.\end{cases}
\end{equation*}\label{v_angle}}
\end{enumerate}
\end{theorem}

\begin{proof}
Theorem~\ref{thm:hilbert_minisum_single} gives the equivalence of \ref{x_minimizer} and \ref{v_orthogonality}.
We denote the linear space $K_n-\bar{x}$ by $V$. Then it remains to show that
\begin{equation*}
v \in \lr{V^\perp \cap \mc{0}{1}} + \mc{0}{1} \quad\Longleftrightarrow\quad \mnorm{v}\leq \begin{cases} \frac{1}{\cos \alpha} &\text{if \,}0 \leq \alpha \leq \frac{\piup}{4},\\
2\sin \alpha &\text{if \,}\frac{\piup}{4} \leq \alpha \leq \frac{\piup}{2}.\end{cases}
\end{equation*}
We can suppose that $v \neq 0$, since the claim is obviously true if $v=0$. In order to shorten notation, we write $w|_W$ for the orthogonal projection of a vector $w \in \EE^d$ onto a linear space $W \subseteq \EE^d$. Clearly,
\begin{align*}
v&=v|_{V^\perp}+v|_V,&\mnorm{v}^2&= \mnorm{v|_{V^\perp}}^2+\mnorm{v|_V}^2,\\
\sin \alpha&= \frac{\mnorm{v|_{V^\perp}}}{\mnorm{v}},&\cos \alpha&= \frac{\mnorm{v|_{V}}}{\mnorm{v}}.
\end{align*}

\emph{I. Proof of `$\Rightarrow$'. } By $v \in \lr{V^\perp \cap \mc{0}{1}} + \mc{0}{1}$, there is a representation
\begin{equation}
\begin{array}{c}
v=u+w,\quad u=u|_{V^\perp} \in V^\perp,\quad \mnorm{u} \leq 1,\quad w=w|_{V^\perp}+w|_V,\\[1ex]
\mnorm{w|_{V^\perp}}^2+\mnorm{w|_V}^2=\mnorm{w}^2 \leq 1,\quad v|_{V^\perp}=u+w|_{V^\perp},\quad v|_V=w|_V.
\end{array}\label{eq_y1}
\end{equation}

\emph{Case 1: $0 \leq \alpha \leq \frac{\piup}{4}$. } We obtain $0 < \mnorm{v|_V} \leq 1$, because
\begin{equation*}
0 < \mnorm{v} \cos \alpha = \mnorm{v|_V}= \mnorm{w|_V} \leq 1.
\end{equation*}
This gives $\mnorm{v} \leq \frac{\mnorm{v}}{\mnorm{v|_V}} = \frac{1}{\cos \alpha}$.

\emph{Case 2: $\frac{\piup}{4} \leq \alpha \leq \frac{\piup}{2}$. } Then 
\begin{equation*}
\mnorm{v|_V} =\mnorm{v} \cos \alpha \leq \mnorm{v} \sin \alpha = \mnorm{v|_{V^\perp}}.
\end{equation*}

\emph{Subcase 2.1: $\mnorm{v|_{V^\perp}} \leq 1$. } Now
\begin{equation*}
\mnorm{v}^2= \mnorm{v|_{V^\perp}}^2+\mnorm{v|_V}^2 \leq 2\mnorm{v|_{V^\perp}}^2 \leq 2\mnorm{v|_{V^\perp}},
\end{equation*}
which yields $\mnorm{v} \leq \frac{2\mnorm{v|_{V^\perp}}}{\mnorm{v}}= 2\sin \alpha$.

\emph{Subcase 2.2: $\mnorm{v|_{V^\perp}}> 1$. } We define $\tilde{u}= \frac{v|_{V^\perp}}{\mnorm{v|_{V^\perp}}}$ and $\tilde{w}=v-\tilde{u}$. Note that
\begin{equation*}
\mnorm{\tilde{w}|_{V^\perp}}=\mnorm{v|_{V^\perp}-\tilde{u}}=
\mnorm{v|_{V^\perp}}-1 \leq \mnorm{u}+\mnorm{w|_{V^\perp}}-1 \leq \mnorm{w|_{V^\perp}}
\end{equation*}
and $\tilde{w}|_V=v|_V-\tilde{u}|_V=v|_V=w|_V$. Therefore, the representation $v=\tilde{u}+\tilde{w}$ satisfies analogous conditions as in \eqref{eq_y1}, but with $\mnorm{\tilde{u}}=1$.
We have
\begin{equation*}
v=\tilde{u}+\tilde{w}= (\tilde{u}+\tilde{w})|_{\lin\setn{v}}= \tilde{u}|_{\lin\setn{v}}+ \tilde{w}|_{\lin\setn{v}}.
\end{equation*}
Application of Pythagoras' theorem to $0, \tilde{u}|_{\lin\setn{v}}, \tilde{u}$ and to
$v, \tilde{u}|_{\lin\setn{v}}, \tilde{u}$ gives
\begin{equation*}
\begin{array}{c}
1=\mnorm{\tilde{u}}^2= \mnorms{\tilde{u}|_{\lin\setn{v}}}^2+\mnorms{\tilde{u}-\tilde{u}|_{\lin\setn{v}}}^2,\\[1ex]
\mnorm{\tilde{w}}^2= \mnorm{v-\tilde{u}}^2= \mnorms{\tilde{u}-\tilde{u}|_{\lin\setn{v}}}^2+\mnorms{v-\tilde{u}|_{\lin\setn{v}}}^2.
\end{array}
\end{equation*}
This yields
\begin{equation}
\mnorms{\tilde{w}|_{\lin\setn{v}}} \leq \mnorms{\tilde{u}|_{\lin\setn{v}}},\label{eq_y2}
\end{equation}
because
\begin{equation*}
\begin{array}{c}
\hspace{-12ex} \mnorms{\tilde{w}|_{\lin\setn{v}}}^2 = \mnorms{v-\tilde{u}|_{\lin\setn{v}}}^2=
\mnorm{v-\tilde{u}}^2-\mnorms{\tilde{u}-\tilde{u}|_{\lin\setn{v}}}^2\\[1ex]
\hspace{5ex}  =\mnorm{\tilde{w}}^2-(1-\mnorms{\tilde{u}|_{\lin\setn{v}}}^2) \leq
1-(1-\mnorms{\tilde{u}|_{\lin\setn{v}}}^2)=\mnorms{\tilde{u}|_{\lin\setn{v}}}^2.
\end{array}
\end{equation*}
The angle between $\lin\setn{v}$ and $\lin\setn{v|_{V^\perp}}$ is $\frac{\piup}{2}-\alpha$. Therefore, the vectors $\tilde{u}|_{\lin\setn{v}} \in \lin\setn{v}$ and $\tilde{u} \in \lin\setn{v|_{V^\perp}}$ are related by
\begin{equation}
\mnorms{\tilde{u}|_{\lin\setn{v}}}=\mnorms{\tilde{u}} \cos\lr{\frac{\piup}{2}-\alpha}= \sin \alpha.\label{eq_y3}
\end{equation}
Finally, \eqref{eq_y2} and \eqref{eq_y3} give our claim
\begin{equation*}
\mnorm{v}=\mnorms{\tilde{u}|_{\lin\setn{v}}+\tilde{w}|_{\lin\setn{v}}} \leq \mnorms{\tilde{u}|_{\lin\setn{v}}}+\mnorms{\tilde{w}|_{\lin\setn{v}}} \leq 2\mnorms{\tilde{u}|_{\lin\setn{v}}} = 2\sin \alpha.
\end{equation*}

\emph{II. Proof of `$\Leftarrow$'. }

\emph{Case 1: $0 \leq \alpha \leq \frac{\piup}{4}$. } The assumption $\mnorm{v} \leq \frac{1}{\cos \alpha}$ implies
\begin{equation*}
\mnorm{v|_V}= \mnorm{v} \cos \alpha \leq \frac{1}{\cos \alpha} \cos \alpha=1.
\end{equation*}
Moreover, by $0 \leq \alpha \leq \frac{\piup}{4}$,
\begin{equation*}
\mnorm{v|_{V^\perp}}= \mnorm{v} \sin \alpha \leq \mnorm{v} \cos \alpha= \mnorm{v|_V} \leq 1.
\end{equation*}
Therefore, $v|_{V^\perp} \in V^\perp \cap \mc{0}{1}$ and $v|_V \in \mc{0}{1}$, and the representation $v=v|_{V^\perp}+v|_V$ shows that $v \in \lr{V^\perp \cap \mc{0}{1}}+\mc{0}{1}$.

\emph{Case 2: $\frac{\piup}{4} \leq \alpha \leq \frac{\piup}{2}$. } We define $u= \frac{v|_{V^\perp}}{2 \sin^2 \alpha} \in V^\perp$ and $w=v-u$. By the assumption $\mnorm{v} \leq 2 \sin \alpha$,
\begin{equation*}
\mnorm{u}=\frac{1}{2 \sin^2 \alpha} \mnorm{v|_{V^\perp}} = \frac{1}{2 \sin^2 \alpha} \mnorm{v} \sin \alpha \leq 1.
\end{equation*}
Moreover,
\begin{align*}
\mnorm{w}^2 &= \mnorm{v-u}^2\\
&=\mnorm{ \lr{v|_{V^\perp} - \frac{v|_{V^\perp}}{2 \sin^2 \alpha}} + v|_V}^2\\
&=\lr{1-\frac{1}{2 \sin^2 \alpha}}^2 \mnorm{v|_{V^\perp}}^2+\mnorm{v|_V}^2\\
&=\lr{1-\frac{1}{2 \sin^2 \alpha}}^2 (\mnorm{v} \sin \alpha)^2 + (\mnorm{v} \cos \alpha)^2\\
&=\frac{\mnorm{v}^2}{(2 \sin \alpha)^2} \\
&\leq 1. 
\end{align*}
Thus, the representation $v=u+w$ shows that $v \in \lr{V^\perp \cap \mc{0}{1}} + \mc{0}{1}$.\qed
\end{proof}

Proposition~3.1 from \cite{KupitzMaSp2013} discusses the question when the Fermat--Torricelli objective function with respect to one affine flat and finitely many points in $\EE^d$ has more than one miminizer. Since the characterization given there is incorrect, we address that problem again. 
\pagebreak
\begin{definition}[{see \cite[Definition~3.1]{KupitzMaSp2013}}]
\begin{enumerate}[label={(\alph*)},leftmargin=0.3cm,align=left]
\item{A collection of $n \geq 2$ distinct points $p_1,\ldots,$ $p_n \in \EE^d$ is called \emph{ortho-collinear to an affine flat $F \subseteq \EE^d$} iff $p_1,\ldots,p_n$ are collinear such that $\aff\setn{p_1,\ldots,p_n} \perp F$ (that is, $\skpr{x-y}{v-w}=0$ for all $x,y \in\aff\setn{p_1,\ldots,p_n}$ and for all $v,w\in F$) and $\aff\setn{p_1,\ldots,p_n} \cap F \neq \emptyset$.}
\item{The \emph{median} of an odd number of collinear points $p_1,\ldots,p_n \in \EE^d$, $n \geq 3$, is the $\frac{n+1}{2}$th point if $p_1,\ldots,p_n$ are naturally ordered along the straight line $\aff\setn{p_1,\ldots,p_n}$.}
\end{enumerate}
\end{definition}

\begin{proposition}\label{correction3.1}
Let $F \subseteq \EE^d$ be an affine flat, $F \neq \emptyset$, let $p_1,\ldots,p_n \in \EE^d$ be $n \geq 1$ distinct points, and let $f: \EE^d \to \RR, f(x)=\dist(x,F)+\sum_{i=1}^n \mnorm{p_i-x}$. Then $f$ has more than one minimum point if and only if one of the following is satisfied.
\begin{enumerate}[label={(\roman*)},leftmargin=0.3cm,align=left]
\item{$n=1$ and $p_1 \notin F$,\label{p_notin_F}}
\item{$n \geq 2$ is even and $\aff\setn{p_1,\ldots,p_n}$ is a straight line contained in $F$,\label{even_case}}
\item{$n \geq 3$ is odd, $p_1,\ldots,p_n$ are ortho-collinear to $F$, and $F$ does not contain the median of $p_1,\ldots,p_n$.\label{odd_case}}
\end{enumerate}
\end{proposition}

We shall use the following lemma.

\begin{lemma}
Let $p_1,\ldots,p_n \in \EE^d$ be $n \geq 1$ distinct points. If $L \subseteq \EE^d$ is a straight line satisfying $\setn{p_1,\ldots,p_n} \centernot \subseteq L$, then the restriction of $h: \EE^d \to \RR$, $h(x)= \sum_{i=1}^n \mnorm{p_i-x}$, to $L$ is strictly convex. 
\end{lemma}

\begin{proof}
Assume that $h|_L$ is not strictly convex. Then there exist $x_1,x_2 \in L$, $x_1 \neq x_2$, such that $h\lr{\frac{x_1+x_2}{2}} \geq \frac{h(x_1)+h(x_2)}{2}$, which is equivalent to
\begin{equation*}
\sum_{i=1}^n \frac{\mnorm{(p_i-x_1)+(p_i-x_2)}}{2} \geq \sum_{i=1}^n \frac{\mnorm{p_i-x_1}+\mnorm{p_i-x_2}}{2}.
\end{equation*}
Consequently, we have equalities in the triangle inequalities 
\begin{equation*}
\mnorm{(p_i-x_1)+(p_i-x_2)} \leq \mnorm{p_i-x_1}+\mnorm{p_i-x_2}
\end{equation*}
for $i\in\setn{1,\ldots,n}$. Hence $p_i-x_1$ and $p_i-x_2$ are linearly dependent, and $p_i \in \aff\setn{x_1,x_2}=L$. This gives $\setn{p_1,\ldots,p_n} \subseteq L$, a contradiction.\qed
\end{proof}

\begin{proof}[of Proposition~\ref{correction3.1}]
\emph{Case 1: $n=1$. } If $p_1 \notin F$, then the segment $[p_1,\Proj(p_1,F)]$ consists of minimizers of $f$, as can be seen directly or shown by Theorem~\ref{thm-Euclidean}. Of course, if $p_1 \in F$, then $p_1$ is the only minimum point of $f$.

\emph{Case 2: $n \geq 2$. } Suppose that $x_1, x_2$ are two distinct minimum points of $f$. 

Assume for a moment that $\setn{p_1,\ldots,p_n} \centernot\subseteq \aff\setn{x_1,x_2}$. Then the last lemma says that $h|_{\aff\setn{x_1,x_2}}$ is strictly convex, and in turn the restriction of the function $f=h+\dist(\cdot,F)$ to $\aff\setn{x_1,x_2}$ is strictly convex as well. But then $x_1$ and $x_2$ cannot minimize $f$ simultaneously. This contradiction yields $\setn{p_1,\ldots,p_n} \subseteq \aff\setn{x_1,x_2}$, and
\begin{equation*}
\aff\setn{p_1,\ldots,p_n} = \aff\setn{x_1,x_2}
\end{equation*}
is a straight line.

\emph{Case 2.1: $\setn{x_1,x_2} \subseteq F$. } Now the straight line $\aff\setn{p_1,\ldots,p_n}$ is contained in $F$. If $n$ were odd, then the median of $p_1,\ldots,p_n$ would be the only minimizer of $h$ (as can be seen directly or shown by Theorem~\ref{thm-Euclidean}) and in turn also of $f$. This contradiction shows that $n$ is even, and \ref{even_case} is verified.

Conversely, if \ref{even_case} is satisfied, then the Fermat--Torricelli locus consists of the segment between the $\frac{n}{2}$th and the  $\lr{\frac{n}{2}+1}$st point of $p_1,\ldots,p_n$ in their
natural order along $\aff\setn{p_1,\ldots,p_n}$, and $f$ has more than one minimum point.

\emph{Case 2.2: $\setn{x_1,x_2} \centernot\subseteq F$. } Since the Fermat--Torricelli locus is convex (see Proposition~\ref{prop-locus_general}), the whole segment $[x_1,x_2]$ consists of minimum points of $f$. Hence there is no loss of generality if we assume that $x_1 \notin F \cup \setn{p_1,\ldots,p_n}$. Then, by Theorem~\ref{thm-Euclidean}\ref{floating}, 
\begin{equation}
\frac{\Proj(x_1,F)-x_1}{\mnorm{\Proj(x_1,F)-x_1}}+ \sum_{i=1}^n \frac{p_i-x_1}{\mnorm{p_i-x_1}}=0,\label{eq-Euclidean}
\end{equation}
and
\begin{align*}
\sum_{i=1}^n \frac{x_1-p_i}{\mnorm{x_1-p_i}}=\frac{\Proj(x_1,F)-x_1}{\mnorm{\Proj(x_1,F)-x_1}} &\in (F-F)^\perp \setminus \setn{0},\\
x_1+ \mnorm{\Proj(x_1,F)-x_1} \sum_{i=1}^n \frac{x_1-p_i}{\mnorm{x_1-p_i}}&=
\Proj(x_1,F)\\&\in \aff\setn{p_1,\ldots, p_n,x_1} \cap F.
\end{align*}
The first equation shows that the direction of the straight line
\begin{equation*}
\aff\setn{p_1,\ldots,p_n}=\aff\setn{p_1,\ldots, p_n,x_1}
\end{equation*}
is in $(F-F)^\perp$. The second one gives $\aff\setn{p_1,\ldots,p_n} \cap F \neq \emptyset$. Thus $p_1,\ldots,p_n$ are ortho-collinear to $F$.
Moreover, in \eqref{eq-Euclidean} the zero vector is represented as sum of $n+1$ vectors of unit length parallel to $\aff\setn{p_1,\ldots,p_n}$. Hence $n+1$ is even, and $n$ is odd. Finally, if $F$ contained the median of $p_1,\ldots,p_n$, then this median would be the unique minimizer of $f$, since it would be the unique minimizer of $h$ and a minimizer of $\dist(\cdot,F)$. This contradiction completes the verification of \ref{odd_case}.

Conversely, suppose that \ref{odd_case} is satisfied. Let $\setn{p_0}=\aff\setn{p_1,\ldots,p_n} \cap F$ and let $q_1,\ldots,q_{n+1}$ denote the points $p_0,\ldots,p_n$ according to their natural order along
the straight line $\aff\setn{p_1,\ldots,p_n}$. If $p_0=p_i$ for some $i\in \setn{1,\ldots,n}$, then this point is represented by two indices: $p_0=p_i=q_j=q_{j+1}$. The set of minimizers of $f$ is given by the segment $ \clseg{q_{\frac{n+1}{2}}}{q_{\frac{n+3}{2}}}$, because $\aff\setn{p_1,\ldots,p_n} \perp F$. Since $p_0$ is not the median of $p_1,\ldots,p_n$, we obtain $q_{\frac{n+1}{2}} \neq q_{\frac{n+3}{2}}$, and $f$ has infinitely many minimum points.\qed
\end{proof}

\section{Perspectives and Interesting Open Problems}
\label{chap:perspectives}
It is clear that one can go on with generalizing minsum location problems based on our results in a very natural way. We list now different possibilities for further research in such directions. As the reader will observe,  there are various natural overlaps between these problems and topics.

We have considered the case where the dimension of $X$ is finite. This way we could use arguments of compactness, and there was only one vector space topology on $X$. But what happens if the dimension of $X$ is infinite? What concepts and statements are preserved, and what is different? And what are appropriate additional assumptions to obtain results similar to the finite-dimensional setting?

In the classical Fermat--Torricelli problem, we have a nice characterization of the case where the solution set is unique. What happens in the generalized setting(s)? Can one characterize the cases where we have an empty solution set, a unique solution, a bounded solution set, an affine flat as solution set, etc.? What is the affine dimension of the solution set? Can one characterize other geometric properties of the solution set? A complete characterization of the Fermat--Torricelli locus of three Euclidean balls with distinct centers in Euclidean space can be found in \cite[Sect.~4.1]{NamHoAn2014}. There is also a short discussion, why we do not need \emph{maximal time functions} in the Euclidean setting, for Euclidean balls. It should be checked whether this remains valid in our setting if $K_i$ is a homothet of $\mcg{\gamma_i}{0}{1}$. However, incorporating maximal time functions is still legit in other cases. A general treatment can be found in \cite{NamHo2013}.

Is there a nice geometric counterpart of Theorem~\ref{thm:hilbert_minisum_single}\ref{hilbertcase_c} for the non-Euclidean case? One might expect that the right-hand set works in general, whereas the vector on the left-hand side has to be replaced by some set. Then the criterion would say that the intersection of the sets is non-empty (all this is very vague so far). Since our theorems on optimality conditions are just special cases of Fermat's rule, the question on Theorem~\ref{thm:hilbert_minisum_single}\ref{hilbertcase_c} is in fact: Can we write the subdifferential of the non-Euclidean distance function in terms of non-Euclidean projections? Often the answer is \enquote{yes} (see \cite[pp.~437--444]{MordukhovichNa2011b}). Analogous results for more general subdifferentials than the convex one can be found in \cite{MordukhovichNa2011a}.

One might extend weighted minsum location problems also such that non-positive weights occur. There exists already some related literature \cite{AnNaYe2014,TaoHo1997,NickelDu1997,ChenHaJaTu1992,ChenHaJaTu1998}, but D.C. programming is not as well studied as convex programming.

One could also look at analogues of minsum location problems in Hadamard spaces. These are complete metric spaces with non-positive curvature. Since geodesics are unique in these spaces, one can still use convex analysis and therefore convex programming. But there is even less literature (see, \zB \cite{Banert2014} and the references therein).

It is also natural to study minsum location problems on manifolds. The book \cite{Udriste1994} on convex optimization is related; it has a small chapter on Finsler manifolds.

\section{Conclusions}
\label{chap:conclusions}
Inspired by various well-known results on the Fermat--Torricelli problem considered in finite dimensional real Banach spaces and by its analogues for searched minimizing hyperplanes or spheres, in the present paper this problem is generalized to arbitrary convex distance functions (gauges), and at the same time the participating geometric configuration is extended to families of arbitrary convex sets. This two-fold generalization of existing results is reached via certain preliminary steps (\zB by considering gauges, but still finite point sets as geometric configurations) which are interesting for themselves and therefore presented, too. It turns out that many results known from the literature suitably carry over to this more general setting, but some of them do not. For deriving these new results, basic geometric notions from Banach space theory are extended to gauges, directly yielding interesting questions for further generalizations (for instance, to extend the problem to infinite-dimensional vector spaces with gauges). Hence we pose also a couple of new research problems. From our paper one can read off the power of geometric tools and of methods from convex analysis to obtain far-reaching, but natural generalizations of one of the oldest and most famous problems from continuous location science.

\begin{acknowledgements}
This research was partially worked out when the fourth named author held a visiting professorship at the Faculty of Mathematics of the Otto von Guericke University of Magdeburg, Germany.
\end{acknowledgements}

\end{document}